\documentclass[11pt,a4paper,reqno]{amsart}
\usepackage{amssymb,amsmath,amsthm}
\usepackage[T1]{fontenc}
\usepackage{enumerate}
\usepackage[all]{xy}
\usepackage{fullpage}
\usepackage{comment}
\usepackage{array}
\usepackage{longtable}
\usepackage{stmaryrd}
\usepackage{mathrsfs} 
\usepackage{xcolor}
\usepackage{mathtools}

\usepackage[utf8]{inputenc}
\usepackage[T1]{fontenc}
\usepackage{fancyhdr}

\def\N{\mathbb{N}}

\def\gcd{\mathrm{gcd}}
\def\deg{\mathrm{deg}}

\def\vol{\mathrm{Vol}}
\def\dim{\mathrm{dim}}

\newcommand{\GL}{\mathrm{GL}}
\newcommand{\SL}{\mathrm{SL}}

\newcommand{\SO}{{\mathrm{SO}}}

\newcommand{\rk}{\mathrm{rk}}

\newcommand{\Cone}{\mathrm{Cone}}
\newcommand{\Rmo}{\mathcal R_{\text{monic}}}

\usepackage{hyperref}
\hypersetup{
    colorlinks=true,
    linkcolor=blue,
    filecolor=red,  
    citecolor=green,
    urlcolor=cyan,
    pdftitle={Siegel Rogers},
    pdfpagemode=FullScreen,
    }

\usepackage{cleveref}
\crefformat{section}{\S#2#1#3}
\crefformat{subsection}{\S#2#1#3}
\crefformat{subsubsection}{\S#2#1#3}
\usepackage{enumitem}
\usepackage{tikz}
\usepackage{mathdots}

\newtheorem{theorem}{Theorem}[section]
\newtheorem{lemma}[theorem]{Lemma}
\newtheorem{claim}[theorem]{Claim}

\newtheorem{proposition}[theorem]{Proposition}

\theoremstyle{definition}
\newtheorem{definition}[theorem]{Definition}

\theoremstyle{remark}
\newtheorem{remark}[theorem]{Remark}

\numberwithin{equation}{section}

\DeclareMathOperator{\supp}{supp}

\title{Quantitative Oppenheim Conjecture for Random Quadratic Forms and Optimal Variance Bounds in Function Fields}
\author{Jiyoung Han}
\email{jiyoung.han@pusan.ac.kr}
\address{Department of Mathematics Education, Pusan National University}
\author{Noy Soffer Aranov}
\email{noy.sofferaranov@tugraz.at}
\address{Graz University of Technology, Institute of Analysis and Number Theory, 8010 Graz, Austria}
\begin{document}

\begin{abstract}
    We prove a quantitative version of Oppenheim's conjecture in the function field setting. In order to do so, we compute the higher moments of the Siegel transform. In particular, we find an optimal bound on the variance of the number of lattice points in a set. Moreover, we compute the exact variance of the number of lattice points in a ball, which is of independent interest. 
\end{abstract}
\maketitle
\section{Introduction}
Let $Q$ be a non-degenerate indefinite irrational quadratic form in $n\geq 3$ variables. Oppenheim's conjecture \cite{Opp} states that every irrational non-degenerate indefinite quadratic form $Q$ has a dense image, i.e., $\overline{Q(\mathbb{Z}^n)}=\mathbb{R}$. Margulis \cite{Mar} proved that Oppenheim's conjecture holds by using homogeneous dynamics. In her seminal paper, Ratner \cite{Rat} proved the orbit closure theorem for unipotent flows, which implies Oppenheim's conjecture. 

Quantitative versions of Oppenheim's conjecture, which handle the rate at which the image of $Q$ becomes dense, were also established. 
This direction was originally initiated by Eskin, Margulis, and Mozes \cite{EMM}, who established that for every non-degenerate, indefinite, irrational quadratic form $Q$ of signature $(p,q)$ for $p\geq 3$ and $q\geq 1$, for every fixed interval $I$, and for every large enough $t$, one has 
$$\#\left\{\mathbf{v}\in \mathbb{Z}^n:Q(\mathbf{v})\in I,\Vert \mathbf{v}\Vert\leq t\right\}\sim \lambda_Q\vert I\vert t^{n-2},$$
where $\lambda_Q$ is a constant depending only on $Q$. A few years later, the same authors \cite{EMM05} showed that the same asymptotic formula holds for split quadratic forms, and recently, Wooyeon Kim \cite{Kim} proved that such an asymptotic holds for the 3-dimensional case. In both cases, the quadratic forms have to satisfy some Diophantine condition.

The error term in the above formula was later bounded by Athreya and Margulis \cite{AM} and by Kelmer and Yu \cite{KY} for almost every quadratic form. 
Recently, Kelmer and Yu \cite{KY23} improved the aforementioned error term for generic quadratic forms of signature $(n-1,1)$. These results are deduced from \emph{moment formulae for Siegel transforms} about the various set-ups. For further historical remarks about Siegel transforms and their moments and application to Oppenheim conjecture-typed problems, see \cite{KY23}.

The Oppenheim conjecture has been studied over different fields as well. For example, Borel and Prasad \cite{BP} proved that the Oppenheim conjecture holds over $S$-arithmetic fields whose ground field has characteristic zero. Furthermore, quantitative versions of Oppenheim's conjecture have been established for isotropic forms, that is, forms $Q$ for which there exists some non-zero vector $\mathbf{v}$, such that $Q(\mathbf{v})=0$, in the $S$-arithmetic setting by the first named author, Lim, and Mallahi-Karai \cite{HLM} and for generic quadratic forms of rank $3$ and $4$ by the first named author \cite{Han3-4}. 

Mohammadi \cite{Moh} proved that Oppenheim's conjecture holds over local fields of positive characteristic. We note that since Ratner's orbit closure theorem \cite{Rat} has not been established in the function field setting, Mohammadi's proof closely follows the topological ideas in Margulis' proof of the real Oppenheim conjecture \cite{Mar}, which do not lead to quantitative results. Despite being well studied over fields with zero characteristic, quantitative versions of Oppenheim's conjecture have not been studied over local fields of positive characteristic. In this paper, we fill this gap by proving a quantitative version of the result of Mohammadi in a generic sense. To prove our main theorem, we establish a function field analogue of Rogers' higher moment formula of the Siegel transform \cite{Rogers1955}. First, we briefly introduce the function field setting.
\subsection{Function Fields}
 Let $q$ be the power of an odd prime, let $\mathcal{R}=\mathbb{F}_q[x]$, and let $\mathcal{K}=\mathbb{F}_q(x)$ be the field of fractions of $\mathcal{R}$. Equip $\mathcal{K}$ with an absolute value $\left|\frac{f}{g}\right|=q^{\deg(f)-\deg(g)}$, so that the topological completion of $\mathcal{K}$ with respect to $\vert \cdot \vert$ is $\mathcal{K}_{\infty}=\mathbb{F}_q(\!(x^{-1})\!)$. It is well known that every local field of positive characteristic is isomorphic to some field of the form $\mathcal{K}_{\infty}$ as above, and therefore, this paper will only handle these types of fields. 
 
 The maximal compact order in $\mathcal{K}_{\infty}$ is $\mathcal{O}_{\infty}=\mathbb{F}_q[[x^{-1}]]$ and its maximal ideal is $\mathfrak{m}=x^{-1}\mathcal{O}_{\infty}$. In this paper, $m_{\mathcal{K}_{\infty}}$ denotes the translation-invariant Haar measure on $\mathcal{K}_{\infty}$ satisfying $m_{\mathcal{K}_{\infty}}\left(\mathcal{O}_{\infty}\right)=1$. Let $\operatorname{Vol}$ denote the $n$-dimensional product measure $m_{\mathcal{K}_{\infty}}^n$ on $\mathcal{K}_{\infty}^n$. 
 
 For $\mathbf{v}=(v_1,\dots,v_n)\in \mathcal{K}_{\infty}^n$, we define the norm of $\mathbf{v}$ by $\Vert \mathbf{v}\Vert=\max_{i=1,\dots,n}\vert v_i\vert$. A metric ball in $\mathcal K_\infty^n$ is a set of the form $B(\alpha,r):=\alpha+x^r\mathcal{O}_{\infty}^n$, where $\alpha\in \mathcal{K}_{\infty}^n$ and $r\in \mathbb Z$. The norm $\Vert \cdot \Vert$ satisfies the ultrametric inequality: For $\mathbf{u},\mathbf{v}\in \mathcal{K}_{\infty}^n$, we have
    \begin{enumerate}
        \item \label{lem:UM} $\Vert \mathbf{u}+\mathbf{v}\Vert\leq \max\{\Vert \mathbf{u}\Vert,\Vert\mathbf{v}\Vert\}$;
        \item \label{lem:UM=} If $\Vert \mathbf{u}\Vert\neq \Vert \mathbf{v}\Vert$, then $\Vert \mathbf{u}+\mathbf{v}\Vert=\max\{\Vert \mathbf{u}\Vert,\Vert \mathbf{v}\Vert\}$.
    \end{enumerate}

The ultrametric inequality, in particular \eqref{lem:UM=}, often enables us to compute norms in a highly accurate way. Moreover, in the ultrametric setting, a pair of balls are either disjoint or one is contained in the other. This property enables very accurate lattice point counting bounds (see Proposition \ref{lem:BallErr}). 

\subsection{Quadratic forms and the Space of Lattices} 
We say that two non-degenerate quadratic forms $Q$, $Q'$ on $\mathcal K_\infty^n$ are \emph{equivalent} if there exists $g\in \GL_n(\mathcal K_\infty)$ so that
\[
Q'(\mathbf{v})= Q(g\mathbf v),\quad \forall \mathbf v\in \mathcal K_\infty^n.
\] 
It is known that two non-degenerate quadratic forms over a function field are equivalent if they have the same rank, discriminant and Hasse-Witt invariant (see \cite[Remark 5.12]{EKM} or \cite[V.3.17]{Lam} for example). Moreover $$\operatorname{GL}_n(\mathcal{K}_{\infty})=\bigsqcup_{g\in S}\mathcal{K}_{\infty}^\times g \operatorname{SL}_n(\mathcal{K}_{\infty}),$$
where $S=\left\{\operatorname{diag}\{x^{t_1},x^{t_2},\cdots,x^{t_n}\}:t_i\geq 0,0\leq \sum_{i=1}^nt_i\leq n-1\right\}$ and $$\mathrm{G}:=\SL_n(\mathcal K_\infty)=\left\{g\in \GL_n(\mathcal K_\infty): |\det g|=1\right\}.$$
Thus, up to scalar multiplication and multiplication by a matrix from the set $S$, the space of non-degenerate quadratic forms is a finite disjoint union of $\mathrm{G}$-orbits, where each of them is isomorphic to $\SO(Q_0)\setminus \mathrm{G}$ for some quadratic form $Q_0$ and $$\SO(Q_0)=\{g\in \mathrm{G}: Q_0(g\mathbf v)=Q_0(\mathbf v),\;\forall \mathbf v\in \mathcal K_\infty^n\}.$$

On the other hand, the density $Q(\mathcal R^d)$ of a quadratic form at integral vectors is invariant under the action of $\Gamma:=\SL_d(\mathcal R)$, i.e., if $Q'=Q\circ \gamma$ for some $\gamma\in \Gamma$, then $Q(\mathcal{R}^d)$ is dense if and only if $Q'(\mathcal{R}^d)$ is dense.

Thus, this density is a property which can be defined on the bi-quotient space $\SO(Q_0)\setminus \mathrm{G}/\Gamma$, where $Q_0$ is a fixed quadratic form. Equivalently, this property can be regarded as an $\SO(Q_0)$-invariant property on $\mathcal{L}_n:=\mathrm{G}/\Gamma$, the space of of unimodular $\mathcal R$-lattices, where  $g\Gamma$ corresponds to $g\mathcal R^n$. In this philosophy, we associate the quadratic form $Q=Q_0\circ g$, where $g\in \mathrm{G}$, with the lattice $g\mathcal R^n$.
The advantage of considering $\mathcal L_n$ is that this homogeneous space supports a natural $\mathrm{G}$-invariant probability measure, denoted by $m_{\mathcal{L}_n}$. 
\subsection{Main Results}
Recall that any isotropic quadratic form $Q$ is equivalent to
\[
Q_0(x_1, \ldots, x_n)=2x_1x_n+ Q'(x_2, \ldots, x_{n-1}),
\]
where $Q'$ is some non-degenerate (on $\mathcal{K}_{\infty}^{n-2}$) quadratic form with $n-2$ variables with coefficients in $\mathcal R$.
Our first main result is a quantification of the result of \cite{Moh} for generic non-degenerate isotropic quadratic forms.
\begin{theorem}
\label{thm:N_QCount}
Let $n\ge 3$ and let $I\subseteq \mathcal{K}_{\infty}$ be a bounded measurable subset. For every $\frac{1}{2}<\delta<1$ and for almost every non-degenerate isotropic quadratic form $Q$, there exists $C_Q>0$ depending only on $Q$, such that
\begin{equation}\label{eqn:Q(v)inICount}\#\left\{\mathbf{v}\in \mathcal{R}^n:Q(\mathbf{v})\in I,\Vert \mathbf{v}\Vert\leq q^t\right\}=C_Qq^{(n-2)t}m_{\mathcal{K}_{\infty}}(I)+O_{Q,I}\left(q^{\delta t(n-2)}\right).\end{equation}
\end{theorem}
To prove Theorem \ref{thm:N_QCount}, we compute the volume of the set $\{\mathbf{v}\in \mathcal{K}_{\infty}^n:Q(\mathbf{v})\in I,\Vert \mathbf{v}\Vert\leq q^t\}$, establish a formula for the higher moments of the Siegel transform,  and then apply this formula to obtain \eqref{eqn:Q(v)inICount}.

\begin{theorem}
    \label{thm:VolEstimate}
    Let $n\ge 3$.
    Let $Q$ be a non-degenerate isotropic quadratic form, and let $I\subseteq \mathcal{K}_{\infty}$ be a one-dimensional ball. There exists $J_Q>0$, depending only on $Q$, such that for every $t>t_{Q,I}$, we have \begin{equation}\label{eqn:volLevelSet}\operatorname{Vol}\left\{\mathbf{v}\in \mathcal{K}_{\infty}^n:Q(\mathbf{v})\in I,\Vert \mathbf{v}\Vert=q^t\right\}=J_Qq^{(n-2)t}m_{\mathcal{K}_{\infty}}(I).\end{equation}
    
    Moreover, there exists $t_{Q,I}>0$, such that for every $t>t_{Q,I}$, we have 
    $$\operatorname{Vol}\left\{\mathbf{v}\in \mathcal{K}_{\infty}^n:Q(\mathbf{v})\in I,\Vert \mathbf{v}\Vert\leq q^t\right\}=C_Qq^{(n-2)t}m_{\mathcal{K}_{\infty}}(I)+O_{Q,I}(1),$$
where $C_Q>0$ is defined as in Theorem~\ref{thm:N_QCount}. Furthermore, we have $C_Q=J_Q\sum_{m=0}^{\infty}q^{-(n-2)m}$.
\end{theorem}
\begin{remark}
    Unlike the real setting, where the volume estimates as in the left hand side of \eqref{eqn:volLevelSet} are only asymptotic \cite{KY,HLM,Han3-4}, in the function field setting, the volume of such sets can be computed exactly, as in \eqref{eqn:volLevelSet}.
\end{remark}
To deduce Theorem \ref{thm:N_QCount} from Theorem \ref{thm:VolEstimate}, we define the function field analogue of the Siegel transform, which enables lattice point counting claims. Towards this end, we say that a function $f:\mathcal{K}_{\infty}^n\rightarrow \mathbb{C}$ is \emph{smooth} if there exists $\varepsilon>0$, such that $f$ is constant on every ball of radius $\varepsilon$. 
\begin{definition}[Siegel Transform]
    For $f:\mathcal{K}_{\infty}^n\rightarrow \mathbb{C}$ smooth with compact support, define the \emph{Siegel transform of $f$}, $\widetilde{f}:\mathcal{L}_n\rightarrow \mathbb{C}$, by $$\widetilde{f}(\Lambda)=\sum_{\mathbf{v}\in \Lambda\setminus\{0\}}f(\mathbf{v}).$$
    For example, if $f$ is the indicator function of a set $B\subseteq \mathcal{K}_{\infty}^n$, then, $\widetilde{f}(\Lambda)=\#B\cap (\Lambda\setminus \{0\})$. 
\end{definition}
We first establish a function field analogue of Siegel's Lemma \cite{Sie}.
\begin{theorem}[Function Field Analogue of Siegel's Lemma]
\label{thm:Siegel}
Let $n\geq 2$ and let $f:\mathcal{K}_{\infty}^n\rightarrow \mathbb{C}$ be a smooth function. Then 
\begin{equation}
\label{eqn:SiegelFormula}\int_{\mathcal{L}_n}\widetilde{f}dm_{\mathcal{L}_n}=q^n\int_{\mathcal{K}_{\infty}^n}fd\operatorname{Vol}.\end{equation}
\end{theorem}
For a fixed set $B$, define a random variable $X_B=\#B\cap (\Lambda\setminus\{0\})$, where $\Lambda\in \mathcal{L}_n$. Theorem \ref{thm:Siegel} states from a probabilistic point of view that $\mathbb{E}[X_B]=q^n\operatorname{Vol}(B)$, and 
$$\operatorname{Var}(X_B)=\int_{\mathcal{L}_n}\widetilde{\mathbf{1}}_B^2dm_{\mathcal{L}_n}-q^{2n}\operatorname{Vol}(B)^2.$$ 
To deduce Theorem \ref{thm:N_QCount} from Theorem \ref{thm:VolEstimate}, we need the following analogue of Rogers' second moment formula \cite{Rogers1955} pertaining to the variance of $X_B$. 
\begin{theorem}
\label{thm:Rogers}
Let $n\ge 3$ and let $q$ be a prime power.
    For every measurable set $B\subseteq \mathcal{K}_{\infty}^n$, we have
    \begin{equation*}
        \operatorname{Var}(X_B)\le C_n\operatorname{Vol}(B),
    \end{equation*}
    where
    \begin{equation}
    \label{eqn:C_n}
    C_n=q^n(2q-1)(1-q^{-1})\frac{q^{n-2}}{q^{n-2}-1}.
    \end{equation}
\end{theorem}

We note that the upper bound $C_n$ is optimal, by considering the case where $B$ is a ball $B(\alpha, r)$ with $\|\alpha\|\le r$. More precisely, one can compute the exact value of $\operatorname{Var}(X_B)$ when $B$ is a ball.
\begin{proposition}
\label{lem:BallErr}
    Let $\alpha\in \mathcal{K}_{\infty}^n$ and let $r>0$. Then, we have
    $$\operatorname{Var}(X_{B(\alpha,r)})=\frac{q^{2n-2}}{q^{n-2}-1}(1-q^{-1})r^n\cdot \begin{cases}
            2q-1&r\geq \Vert \alpha\Vert\\
            \left(\frac{r}{\Vert \alpha\Vert}\right)^{n-1}&r<\Vert \alpha\Vert
        \end{cases}.$$
\end{proposition}
 In the real setting, it is easy to compute
\[
\operatorname{Var}\left(X_{B(0,r)}\right)=\left(4\frac {\zeta(d-1)} {\zeta(d)} - 2\right)\vol(B(0,r))
\]
for balls centered at the origin (see for instance the computation in the proof of \cite[Theorem 2.2]{AM2009}). However, it seems difficult to find the exact variance for balls not centered at the origin. 
\vspace{0.1in}
An important component of the proof of Theorem \ref{thm:Siegel} and Theorem \ref{thm:Rogers} is proving that the Margulis-$\alpha$ function has bounded moments. This can be viewed as a function field analogue of \cite[Lemma 3.10]{EMM}. 
\begin{definition}[Margulis $\alpha$-function]
    Let $0\leq k\leq n$. For a $k$-sublattice $\Delta$ of $\Lambda$, define $L(\Delta)=\Vert \mathbf{v}_1\wedge\dots\wedge\mathbf{v}_k\Vert$, where $\Delta=\mathcal{R}\mathbf{v}_1+\dots+\mathcal{R}\mathbf{v}_k$, and $\mathbf{v}_1,\dots,\mathbf{v}_k\in\mathcal{K}^n_{\infty}$ are linearly independent. Note that $L(\Delta)$ represents the $q^{k}$-multiple of the covolume of $\Delta\subseteq \mathcal K_\infty.\Delta$.  Define
    $$\alpha_k(\Lambda)=\max\left\{\frac{1}{L(\Delta)}:\Delta\text{ is a }k\text{-sublattice of }\Lambda\right\}.$$
    Define the \emph{Margulis $\alpha$-function} by
    $$\alpha(\Lambda)=\max_{k=0,\dots,n}\alpha_k(\Lambda).$$
\end{definition}
\begin{theorem}
\label{lem:alphaMoment}
    For every $1\leq m  < n$, the $m$-th moment of $\alpha$ is finite, i.e.,
    $$\int_{\mathcal{L}_n}\alpha^mdm_{\mathcal{L}_n}<\infty.$$
\end{theorem}
\begin{remark}
Theorem \ref{thm:Siegel}, Theorem \ref{thm:Rogers} and Lemma \ref{lem:BallErr} hold over all positive characteristics, whereas we assume that $q$ is the power of an odd prime in Theorem~\ref{thm:N_QCount}.
    When the ground field has characteristic $2$, a quadratic form does not correspond to a bilinear form, since $2$ is not invertible, which makes Theorem \ref{volume estimate} fail in this case; as does Theorem \ref{thm:N_QCount}.  
\end{remark}
\subsection{Structure of the Paper} The paper is organized as follows: In \cref{sec:GON}, we discuss subgroups of $\mathrm{G}$ in depth and prove Theorem \ref{lem:alphaMoment}. We then use the results of \cref{sec:GON} to prove Theorem \ref{thm:Siegel} and Theorem \ref{thm:Rogers} in \cref{sec:SiegelRogers}. In \cref{sec:Oppenheim}, we prove Theorem \ref{thm:VolEstimate} and apply Theorems \ref{thm:Siegel} and \ref{thm:Rogers} to prove Theorem \ref{thm:N_QCount}. The proof of Theorem \ref{thm:VolEstimate} relies on Theorem \ref{thm: Witt}, which we prove in \cref{sec:Witt}. 
\subsection{Acknowledgements}
The authors thank the organizers of the conference at York "Diophantine approximation and related fields" for organizing the conference where this project began due to the talk of the first named author. JH is supported by the National Research Foundation of Korea (NRF) grant funded by the Korean government (MSIT) (RS-2025-00515082; RS-2026-25478488). NSA is supported by a NAWI-Mitteln Start-Paket granted to new faculty members by NAWI Graz. 
\section{Bounded Moments of the Margulis-$\alpha$ Function and Lattice Point Counting}
\label{sec:GON}
In this section, we provide results about lattice point counting and geometry of numbers in $\mathcal{K}_{\infty}^n$, which are necessary to prove Theorem \ref{lem:alphaMoment}. Moreover, these results enable us to deduce Theorems \ref{thm:Siegel} and \ref{thm:Rogers} from Theorem \ref{lem:alphaMoment}. For more extensive surveys about geometry of numbers in the function field setting, see \cite{Mah,BT,Weil,RW,KST,A}. We first define some groups in $\mathrm{G}$ and their corresponding Haar measures.

Let $A<\mathrm{G}$ be the group of diagonal matrices and let $N<\mathrm{G}$ be the group of upper triangular matrices with $1$'s on the diagonal. These groups $N$, $A$, and $\operatorname{SL}_n(\mathcal{O}_{\infty})$ support Haar measures which are denoted by $m_N$, $m_A$, and $m_{\mathcal{O}_{\infty}}$, respectively. Let $A(\mathcal{O}_{\infty})$ be the group of diagonal matrices whose diagonal elements belong to $\mathcal{O}_{\infty}$. It follows that $A$ is topologically isomorphic to $A(\mathcal{O}_{\infty})\times \mathbb{Z}_0^n$, where 
$$\mathbb{Z}_0^n=\left\{\mathbf k=(k_1,\dots,k_n)\in \mathbb{Z}^n:\sum_{i=1}^nk_i=0\right\},$$
through the map $(\mathbf{o},\mathbf{k})\mapsto \operatorname{diag}\{x^{k_1}o_1,\dots,x^{k_n}o_n\}=:\mathbf{o}\cdot x^{\mathbf{k}}$.
If we let the measure $m_A$ satisfy $m_A(A(\mathcal{O}_{\infty}))=1$, then it decomposes as follows \cite{Casselman}:
\begin{equation}\label{eqn:AmeaDecomp}
dm_A(\mathbf{a})=\prod_{i=1}^{n-1}\frac{da_i}{\vert a_i\vert^{2(i-n)}}=q^{-\sum_{i=1}^{n-1}2(i-n)k_i}dk_1\cdots dk_{n-1}do_1\cdots do_{n-1}.
\end{equation}

In $\mathrm{G}$, we have the Iwasawa decomposition \cite{BT,Casselman,A}, which states that $G=\operatorname{SL}_n(\mathcal{O})AN$. The decomposition of the invariant measure in $\mathcal{L}_n$ with respect to the Iwasawa decomposition is $dm_{\mathcal{L}_n}(o\mathbf{a}u\mathcal{R}^n)=dm_{\operatorname{SL}_n(\mathcal{O})}(o)dm_A(\mathbf{a})dm_N(u)$ \cite{Casselman,PR,BT}. Since $A\cong A(\mathcal{O}_{\infty})\times \mathbb{Z}_0^n$, by the Iwasawa decomposition, we may assume that $\mathbf{a}=\operatorname{diag}\left\{x^{k_1},\dots,x^{k_n}\right\}$. Thus, by \eqref{eqn:AmeaDecomp}, we have
\begin{equation}\label{eqn:AmeaDecompNewCoors}
dm_A(\mathbf{a})=q^{-2\sum_{i=1}^{n-1}(i-n)k_i}dk_1\cdots dk_{n-1}.
\end{equation}
If $\Lambda=g\mathcal{R}^n$, a convenient fundamental domain of $\Lambda$ is $g\mathfrak{m}^n$ and thus $\operatorname{covol}(\Lambda)=q^{-n}\vert \det(g)\vert$. For the sake of simplicity, let us denote $$\det(\Lambda)=\vert \det(g)\vert.$$ Mahler \cite[Equations (24) and (25)]{Mah} proved the following version of Minkowski's 2nd theorem.
\begin{theorem}
    \label{thm:Mink2nd}

    Let $\Lambda=g\mathcal{R}^n\subseteq \mathcal{K}_{\infty}^n$ be a lattice, where $g\in \operatorname{GL}_n(\mathcal{K}_{\infty})$. For $i=1,\dots,n$, define the $i$-th successive minima of $\Lambda$ by
    $$\lambda_i(\Lambda)=\min\{r\geq 0:\operatorname{rk}\left(\Lambda\cap B(0,r)\right)\geq i\}.$$
    Then we have
    \begin{equation}
        \prod_{i=1}^n\lambda_i(\Lambda)=\det(\Lambda).
    \end{equation}
\end{theorem}
Due to Theorem \ref{thm:Mink2nd}, the Margulis-$\alpha$ function has the following form:
\begin{equation}
\label{eqn:alphaFuncForm}
\alpha(\Lambda)=\prod_{i=1}^k\frac{1}{\lambda_i(\Lambda)}.\end{equation}
Moreover, using the successive minima $\lambda_1(\Lambda),\dots,\lambda_n(\Lambda)$, the second named author \cite{A} used Theorem \ref{thm:Mink2nd} to refine the Iwasawa decomposition for $\mathcal{L}_n$.
\begin{theorem}{\cite[Theorem 1.24]{A}}
\label{thm:LattKAN}
    Let $\Lambda\in \mathcal{L}_n$. Then there exists $o\in \operatorname{SL}_n(\mathcal{O})$, such that $\Lambda=o\operatorname{diag}\left\{x^{\log_q\lambda_1(\Lambda)},\dots,x^{\log_q\lambda_n(\Lambda)}\right\}\mathcal{R}^n$.
\end{theorem}
Another application of Theorem \ref{thm:Mink2nd} is counting lattice points in a convex body. Recall that Mahler proved that all convex bodies in this setting are of the form $h\mathcal{O}_{\infty}^n$, where $h\in \GL_d(\mathcal K_\infty)$ \cite{Mah}. We utilize the following lattice point counting claim from \cite{BK}.
\begin{lemma}{\cite[Lemma 6.2]{BK}}
\label{lem:BK}
    For every lattice $\Lambda$ and for every convex body $\mathcal{C}=h\mathcal{O}_{\infty}^n$, we have
    $$\#\Lambda\cap \mathcal{C}=\prod_{i=1}^n\left\lceil \frac{q}{\lambda_i(h^{-1}\Lambda)}\right\rceil,$$
    where $\lceil t\rceil=\min\{n\in \mathbb{Z}:n\geq t\}$. In particular, if $\lambda_{n}(h^{-1}\Lambda)<q$, we have
    $$\#\Lambda\cap \mathcal{C}=\frac{q^n\operatorname{Vol}(\mathcal{C})}{\operatorname{covol}(\Lambda)}.$$
\end{lemma}
We need the following generalization of Lemma \ref{lem:BK} for translates of convex bodies. A convex body $\mathcal{C}$ induces a norm on $\mathcal{K}_{\infty}^n$ defined by
$$\Vert \mathbf{v}\Vert_{\mathcal{C}}=\inf\left\{r\geq 0:\mathbf{v}\in x^{\log_q(r)}\mathcal{C}\right\}.$$
Define $B_{\mathcal{C}}(\mathbf{v},q^r)=\left\{\mathbf{u}\in \mathcal{K}_{\infty}^n:\Vert \mathbf{u}-\mathbf{v}\Vert_{\mathcal{C}}\leq q^r\right\}.$
\begin{lemma}
    Let $\mathbf{v}\in \mathcal{K}_{\infty}^n$ and let $\mathcal{C}=h\mathcal{O}_{\infty}^n$ be a convex body. Let $\Lambda\in \mathcal{L}_n$. There exists $t_0$, depending on $\Lambda$, $\mathbf{v}$, and $\mathcal{C}$, such that for every $t\geq t_0$, we have 
    $$\#\Lambda\cap B_{\mathcal{C}}(\mathbf{v},q^t)=\prod_{i=1}^n\left\lceil\frac{q^{t+1}}{\lambda_i(h^{-1}\Lambda)}\right\rceil.$$
\end{lemma}
\begin{proof}
    Note that it suffices to prove this for $\mathcal{C}=\mathcal{O}_{\infty}^n$, since 
    \begin{equation}
        \Vert \mathbf{v}\Vert_{\mathcal{C}}=\inf\left\{r\geq 0:\mathbf{v}\in x^{\log_qr}h\mathcal{O}_{\infty}^n\right\}=\inf\left\{r\geq 0:h^{-1}\mathbf{v}\in x^{\log_qr}\mathcal{O}_{\infty}^n\right\}=\Vert h^{-1}\mathbf{v}\Vert.
    \end{equation}
    By \cite[Theorem 1.21]{A}, there exists an $\mathcal{R}$-basis $\{\mathbf v_1, \ldots, \mathbf v_n\}$ for $\Lambda$ such that 
    $$\Vert \mathbf{v}_i\Vert=\lambda_i(\Lambda),\quad i=1,\dots,n,$$
    where $\lambda_i(\Lambda)$ is the $i$-th successive minima. Moreover, by combining Theorem~\ref{thm:Mink2nd} and \cite[Theorem 1.21]{A}, we obtain that $\{\mathbf{v}_1,\dots,\mathbf{v}_n\}$ is an orthogonal $\mathcal{R}$-basis of $\Lambda$. Denote by $\mathbf v=\sum_{i=1}^n c_i\mathbf v_i$ ($c_i\in \mathcal K_\infty$). It follows that
    \[\begin{split}
    \mathbf u=\sum_{i=1}^n a_i\mathbf v_i\in \Lambda\cap B(\mathbf v, q^t)
    \quad &\Leftrightarrow\quad \left\|\sum_{i=1}^n (a_i-c_i)\mathbf v_i\right\|\le q^t\\
    &\Leftrightarrow\quad \max_{i=1,\ldots, n} |a_i-c_i|\cdot \lambda_i(\Lambda) \le q^t.
    \end{split}\]
    As a consequence, $\vert a_i-c_i\vert\leq \frac{q^t}{\lambda_i(\Lambda)}$ for every $i$. If $t$ is chosen so that $q^t>\vert c_i\vert\lambda_i(\Lambda)$ for every $i$, we have $\mathbf{u}\in B(\mathbf{v},q^t)$ if and only if $\vert a_i\vert\leq \frac{q^t}{\lambda_i(\Lambda)}$ for every $i$. In this case, we have
    $$\#\Lambda\cap B(\mathbf{v},q^t)=\#\left\{(a_1,\dots,a_n)\in \mathcal{R}^n:\vert a_i\vert\leq \frac{q^t}{\lambda_i(\Lambda)}\right\}=\prod_{i=1}^n\left\lceil\frac{q^{t+1}}{\lambda_i(\Lambda)}\right\rceil.$$
\end{proof}
To prove Theorem \ref{thm:Siegel} and Theorem \ref{thm:Rogers} from Theorem \ref{lem:alphaMoment}, we need a function field analogue of Schmidt's lemma, which states that $\widetilde{f}\ll \alpha$. 

To prove such an analogue, for an operation $*\in \{=,\geq ,>,\leq,<\}$ and $t\in \mathbb{Z}$, define 
     $$\mathcal{L}_n^{*t}=\{\Lambda\in \mathcal{L}_n:\lambda_n(\Lambda)*q^t\}.$$
     By Mahler's compactness criterion \cite[Theorem 1.1]{KST}, $\mathcal{L}_n^{\geq q^t}$ are increasing compact sets, and $\lim_{t\rightarrow \infty}m_{\mathcal{L}_n}\left(\mathcal{L}_n^{>q^t}\right)=0$. Moreover, \cite[Theorem 1.1]{KST} states that a closed set $\mathcal{C}\subseteq \mathcal{L}_n$ is compact if and only if there exists $c>0$ such that $\inf_{\Lambda\in \mathcal{C}}\lambda_1(\Lambda)\geq c$. 

     A function $f:\mathcal{K}_{\infty}^n\rightarrow\mathbb{C}$ is called smooth if there exists $\varepsilon>0$, such that $f$ is constant on every ball of radius $\varepsilon$. We denote the space of smooth functions $f:\mathcal{K}_{\infty}^n\rightarrow \mathbb{C}$ by $C^{\infty}(\mathcal{K}_{\infty}^n)$ and the space of smooth, compactly supported functions $f:\mathcal{K}_{\infty}^n\rightarrow \mathbb{C}$ by $C^{\infty}_c(\mathcal{K}_{\infty}^n)$.
\begin{lemma}[Schmidt's Lemma in Function Fields]\label{Function Field Schmidt Lemma}
\label{cor:LattPointCnt}
    For any $t\in \mathbb{N}$ and for any lattice $\Lambda\subseteq \mathcal{K}_{\infty}^n$, we have
    \begin{enumerate}
        \item \label{cor:LattPointCntComp} $\#\Lambda\cap B(0,q^t)=q^{n(t+1)}$, if $\Lambda\in \mathcal{L}_n^{\leq q^{t+1}}$;
        \item \label{cor:LattPointCntCusp} $\#\Lambda\cap B(0,q^t)\leq q^{n(t+1)}\alpha(\Lambda)$, if $\Lambda\in \mathcal{L}_n^{>q^{t+1}}$.
        \item \label{lem:SchmidtLem} As a consequence, for any $f\in C_c^{\infty}(\mathcal{K}_{\infty}^n)$ there exists $C>0$ depending on $f$, such that for every $\Lambda\in \mathcal{L}_n$, we have 
    $$\widetilde{f}(\Lambda)\leq C\alpha(\Lambda).$$
    \end{enumerate}
\end{lemma}
\begin{proof}
    If $\lambda_n(\Lambda)\leq q^{t+1}$, then, $\left\lceil\frac{q^{t+1}}{\lambda_i(\Lambda)}\right\rceil=\frac{q^{t+1}}{\lambda_i(\Lambda)}$ for every $i=1,\dots,n$. Hence, by Lemma \ref{lem:BK}, 
    \begin{equation}
    \label{eqn:LambdacapB(0,q^t)SmallLambda_n}
        \#\Lambda\cap B(0,q^t)=\prod_{i=1}^n\frac{q^{t+1}}{\lambda_i(\Lambda)}=\frac{q^{(t+1)n}}{\det(\Lambda)}=q^{(t+1)n}\leq q^{t(n+1)}\alpha(\Lambda).
    \end{equation}
    On the other hand, if $\lambda_n(\Lambda)>q^{t+1}$, then there exists some $j=1,\dots,n-1$, such that $\lambda_j(\Lambda)\leq q^{t+1}<\lambda_{j+1}(\Lambda)$. Let $\mathbf{v}_1,\dots,\mathbf{v}_n\in \Lambda$ be a linearly independent set of vectors satisfying $\Vert \mathbf{v}_i\Vert=\lambda_i(\Lambda)$. Then, by Theorem \ref{thm:Mink2nd} and Lemma \ref{lem:BK},
    \begin{equation}
    \label{eqn:LambdacapB(0,q^t)}
        \#\Lambda\cap B(0,q^t)=\prod_{i=1}^j\frac{q^{t+1}}{\lambda_i(\Lambda)}=\frac{q^{(t+1)j}}{\Vert \mathbf{v}_1\wedge \dots \wedge \mathbf{v}_j\Vert}\leq q^{n(t+1)}\alpha(\Lambda).
    \end{equation}
    For \eqref{lem:SchmidtLem}, it suffices to prove that $\widetilde{f}(\Lambda)\leq C\alpha(\Lambda)$ for $f=\mathbf{1}_{B(\mathbf{u},q^t)}$ for $t\in\mathbb{Z}$, since any $f\in C_c^{\infty}(\mathcal{K}_{\infty}^n)$ can be approximated by sums of finitely many indicator functions. Note that in this case, we have
    \begin{equation}
    \label{eqn:tilde(f)Count}
        \widetilde{f}(\Lambda)=\#\Lambda\cap B(\mathbf{u},q^t)=\#\{\mathbf{v}\in \Lambda:\Vert \mathbf{v}-\mathbf{u}\Vert\leq q^{t}\}.
    \end{equation}
    Note that if $\Vert \mathbf{u}\Vert\leq q^t$, then $\Vert \mathbf{v}-\mathbf{u}\Vert\leq q^t$ if and only if $\Vert \mathbf{v}\Vert\leq q^t$. Therefore, in this case, by Lemma \ref{lem:BK}, the right hand side of \eqref{eqn:tilde(f)Count} is equal to
    \begin{equation}
    \label{eqn:BallCntAlphaBND}
        \#\Lambda\cap B(0,q^t)=\prod_{i=1}^n\left\lceil\frac{q^{t+1}}{\lambda_i(\Lambda)}\right\rceil\leq q^{nt}\alpha(\Lambda).
    \end{equation}
    If $\Vert \mathbf{u}\Vert>q^t$, then $\Vert \mathbf{v}-\mathbf{u}\Vert\leq q^t$ implies that $\Vert \mathbf{v}\Vert=\Vert \mathbf{u}\Vert$. As a consequence,  \eqref{eqn:LambdacapB(0,q^t)SmallLambda_n}, \eqref{eqn:LambdacapB(0,q^t)}, and \eqref{eqn:tilde(f)Count} imply that
    \begin{equation*}
        \widetilde{f}(\Lambda)\leq \#\Lambda\cap \left(B(0,\Vert \mathbf{u}\Vert)\setminus B(0,q^{-1}\Vert \mathbf{u}\Vert)\right)=\prod_{i=1}^n\left\lceil \frac{q\Vert \mathbf{u}\Vert}{\lambda_i(\Lambda)}\right\rceil-\prod_{i=1}^n\left\lceil\frac{\Vert \mathbf{u}\Vert}{\lambda_i(\Lambda)}\right\rceil\leq q^n\Vert \mathbf{u}\Vert^n\alpha(\Lambda).
    \end{equation*}
\end{proof}
To conclude this section, we prove Theorem \ref{lem:alphaMoment}.
\begin{proof}[Proof of Theorem \ref{lem:alphaMoment}]
    Note that $\alpha$ is $\operatorname{SL}_n(\mathcal{O}_{\infty})$-invariant. Thus, by \eqref{eqn:alphaFuncForm} and Theorem \ref{thm:LattKAN}, for any $\Lambda\in \mathcal{L}_n$, we have
    \begin{equation}
    \label{eqn:alphaCompDiag}
    \alpha(\Lambda)=\alpha\left(\operatorname{diag}\{x^{\log_q\lambda_1(\Lambda)},\dots,x^{\log_q\lambda_n(\Lambda)}\}\mathcal{R}^n\right)=\max_{j=0,\dots,n}\prod_{1\leq i\leq j}\lambda_i(\Lambda)^{-1}.\end{equation}
    Let 
    $$A_0=\left\{a=\operatorname{diag}\{x^{k_1},\dots,x^{k_n}\}: k_i\leq k_{i+1},k_i\in \mathbb{Z},\,\,\,\forall i=1,\dots,n-1\right\}\subseteq A.$$
     By Theorem \ref{thm:LattKAN}, the set $\operatorname{SL}_n(\mathcal{O}_{\infty})A_0$ contains a fundamental domain for the right action of $\operatorname{SL}_n(\mathcal{R})$ on $\operatorname{SL}_n(\mathcal{K}_{\infty})$. 
    Thus, it suffices to compute
    \begin{equation}
    \label{eqn:IntDecAlpha}
    \begin{split}
    \int_{\mathcal{L}_n}\alpha_j^m\mathrm{d}m_{\mathcal{L}_n}&=\int_{\operatorname{SL}_n(\mathcal{O})A_0}\alpha_j^m(g\mathcal{R}^n)dm_{\mathrm{G}}(g)\\
 &=\int_{\operatorname{SL}_n(\mathcal{O}_{\infty})}\int_{A_0}\alpha_j^m(oa\mathcal{R}^n)\mathrm{d}m_{\operatorname{SL}_n(\mathcal{O})}(o)\mathrm{d}m_A(a)\\
    &\leq C\int_{A_{0}}\alpha_j^m(a\mathcal{R}^n)\mathrm{d}m_A(a),
    \end{split}
    \end{equation}
    where the last inequality stems from compactness of $\operatorname{SL}_n(\mathcal{O}_{\infty})$. Hence, by \eqref{eqn:AmeaDecompNewCoors}, it suffices to compute
    \begin{equation}
    \label{eqn:alphaAInteg}
        \int_{-\infty}^{\infty}\cdots\int_{-\infty}^{\infty}\frac{dk_1\dots d_{k_{n-1}}}{q^{\sum_{i=1}^jmk_i+2\sum_{i=1}^{n-1}(n-i)k_i}}.
    \end{equation}
    For $a=\operatorname{diag}\{x^{k_1},\dots,x^{k_n}\}\in A_0$, denote $\beta_i=k_{i+1}-k_i\geq 0$, so that $k_{i+1}=k_1+\sum_{j=1}^i\beta_j$. Then, $k_1=-\sum_{i=1}^{n-1}k_{i+1}=-(n-1)k_1-\sum_{i=1}^{n-1}\sum_{j=1}^i\beta_j$, so that $k_1=-\frac{1}{n}\sum_{j=1}^{n-1}(n-j)\beta_j$, and for every $i=1,\dots,n-1$, we have 
    \begin{equation}
    \begin{split}
    \label{eqn:k_i+1}
    k_{i+1}=-\frac{1}{n}\sum_{j=1}^{n-1}(n-j)\beta_j+\sum_{j=1}^i\beta_j=\frac{1}{n}\sum_{j=1}^ij\beta_j-\frac{1}{n}\sum_{j=i+1}^{n-1}(n-j)\beta_j.
    \end{split}
    \end{equation}
    Hence by \eqref{eqn:k_i+1}, the exponent of the denominator in \eqref{eqn:alphaAInteg} is equal to 
    \begin{equation}
    \begin{split}
    \label{eqn:denomAlphaIntExp}
        &\frac{1}{n}\left[\sum_{i=1}^jm\left(\sum_{\ell=1}^i\ell\beta_{\ell}-\sum_{\ell=i+1}^{n-1}(n-\ell)\beta_{\ell}\right)+\frac{2}{n}\sum_{i=1}^{n-1}(n-i)\left(\sum_{\ell=1}^i\ell\beta_{\ell}-\sum_{\ell=i+1}^{n-1}(n-\ell)\beta_{\ell}\right)\right]\\
        &=\frac{1}{n}m\left[\sum_{\ell=1}^j\beta_{\ell}\left(n+\ell(j+2-n)\right)-\sum_{\ell=j+1}^{n-1}(n-\ell)\beta_{\ell}\right]\\
        &\hspace{1.2in}+\frac{2}{n}\sum_{\ell=1}^{n-1}\beta_{\ell}\left[\ell\sum_{i=\ell}^{n-1}(n-i)-(n-\ell)\sum_{i=1}^{\ell-1}(n-i)\right]\\
        &=\frac{1}{n}m\left[\sum_{\ell=1}^j\beta_{\ell}\left(n+\ell(j+2-n)\right)-\sum_{\ell=j+1}^{n-1}(n-\ell)\beta_{\ell}\right]+\frac{1}{n}\sum_{\ell=1}^{n-1}\beta_{\ell}\left(3\ell^2+2(n-1)\ell+n^2\right)\\
        &=\frac{1}{n}\sum_{\ell=1}^j\beta_{\ell}\left(mn+m\ell(j+2-n)+3\ell^2+2(n-1)\ell+n^2\right)\\
        &\hspace{1.2in}+\frac{1}{n}\sum_{\ell=j+1}^{n-1}\beta_{\ell}\left(3\ell^2+(2n-2+m)\ell+n(n-m)\right).
    \end{split}
    \end{equation}
    Note that since $m\leq n-1$, then $3\ell^2+(2n-2+m)\ell+n(n-m)\geq 3\ell^2+n\geq 1$ for every $\ell=j+1,\dots,n-1$. Moreover, for $\ell=1,\dots,j$, the coefficient of $\beta_{\ell}$ in the second to last line of \eqref{eqn:denomAlphaIntExp} is equal to
    \begin{equation}\begin{split}\label{eqn:ellSmallExp}
    &mn(1-\ell)+m\ell(j+2)+2(n-1)\ell+n^2+3\ell^2\\
    &\hspace{0.5in}=m\ell(j-n)+mn+2m\ell+2(n-1)\ell+n^2+3\ell^2\\
    &\hspace{0.5in}\geq (n-1)(2\ell-m)+n^2+3\ell^2+2m\ell+mn\\
    &\hspace{0.5in}\geq (n-1)(2-m)+n^2+3\ell^2+2m\ell+mn\\
    &\hspace{0.5in}\geq n\left(n-m+2\right)+3\ell^2+2m\ell+mn>0.
    \end{split}\end{equation}
    As a consequence of \eqref{eqn:denomAlphaIntExp} and \eqref{eqn:ellSmallExp}, there exist $c_1,\dots,c_{n-1}\geq 1$, such that \eqref{eqn:alphaAInteg} is equal to 
    \begin{equation}
        \int_0^{\infty}\cdots \int_0^{\infty}\frac{d\beta_1\dots d\beta_{n-1}}{q^{\sum_{i=1}^n c_i\beta_i}}<\infty.
    \end{equation}
Thus, for every $j=1,\dots,n-1$, and for every $m=1,\dots, n-1$, the integral $\int_{\mathcal{L}_n}\alpha_j^mdm_{\mathcal{L}_n}$ converges. Since $\alpha^m(\Lambda)=\max_{j=1,\dots,n-1}\alpha_j^m(\Lambda)$, the integral $\int_{\mathcal{L}_n}\alpha^mdm_{\mathcal{L}_n}$ converges for every $m=1,\dots,n-1$.
\end{proof}

\section{Siegel's and Rogers' Formulae in Function Fields}
\label{sec:SiegelRogers}
In this section, we prove the first moment (Theorem~\ref{thm:Siegel}) and higher moment (Theorem~\ref{thm:kMomentsForm}) formulae. Moreover, we deduce the upper bound for the variance of the Siegel transform (Theorem \ref{thm:Rogers}) from the first and second moment formulae. 

\subsection{First Moment}
We begin with the following decomposition on the $\operatorname{SL}_n(\mathcal{K}_{\infty})$-invariant measures on $\mathcal K_\infty^n$.
\begin{lemma}
\label{lem:SL_n(K_infty)InvMeas}
    Let $\nu$ be an $\operatorname{SL}_n(\mathcal{K}_{\infty})$-invariant measure on $\mathcal{K}_{\infty}^n$. Then, there exist positive scalars $c_0,c_1\in [0,\infty)$ such that $\nu=c_0\delta_0+c_1\operatorname{Vol}$, where $\delta_0$ is the Dirac-delta mass at $0$.
\end{lemma}
\begin{proof}
    Assume that $H$ acts on a locally compact homogeneous space $X$ with finitely many orbits, $X_1,\dots, X_k$, and $\nu$ is an $H$-invariant measure on $X$. Then, $\nu$ can be decomposed into a sum $\nu=\sum_{i=1}^m\nu_i$, where each $\nu_i$ is of the form $\nu_i(E):=\nu(E\cap X_i)$ and is an $H$-invariant measure supported on $X_i$. In this case, the $\operatorname{SL}_n(\mathcal{K}_{\infty})$-action on $\mathcal{K}_{\infty}^n$ has two orbits, which are $\mathcal{K}_{\infty}^n\setminus \{0\}$ and $\{0\}$ (see \cite{Weil} for more details). An invariant measure supported on $\mathcal{K}_{\infty}^n\setminus\{0\}$ is measure theoretically isomorphic to a scalar multiple of the Lebesgue measure, and an invariant measure supported on $\{0\}$ is a scalar multiple of the Dirac-delta mass at $0$. Hence, there exist $c_0,c_1\geq 0$, such that $\nu=c_0\delta_0+c_1\operatorname{Vol}$.
\end{proof}
\begin{proof}[Proof of Theorem \ref{thm:Siegel}]
    By Lemma \ref{cor:LattPointCnt}\eqref{lem:SchmidtLem} and Theorem \ref{lem:alphaMoment}, the function $\widetilde{f}:\mathcal{L}_n\rightarrow \mathbb{C}$ is integrable. It follows from the Riesz-Markov-Kakutani representation theorem that the linear functional $f\mapsto \int_{\mathcal{L}_n}\widetilde{f}dm_{\mathcal{L}_n}$ corresponds to some $\mathrm{G}$-invariant Radon measure $\nu$, i.e., $$\int_{\mathcal{L}_n}\widetilde{f}dm_{\mathcal{L}_n}=\int_{\mathcal{K}_{\infty}^n}fd\nu
    =c_0f(0)+c_1\int_{\mathcal K_\infty^n} f d\vol$$ 
    for some $c_0,\;c_1\in [0,\infty)$, where the last equality follows from Lemma \ref{lem:SL_n(K_infty)InvMeas}.
    
    To compute $c_1$, consider the sequence of test functions $f_m=\frac{1}{q^{mn}}\mathbf{1}_{x^m\mathcal{O}_{\infty}^n}$, where $m\in \mathbb N$. By Lemma \ref{lem:BK}, we have
      \begin{equation}
    \begin{split}
    \label{eqn:STLargeBalls}
        \int_{\mathcal{L}_n}\widetilde{f}_m dm_{\mathcal{L}_n}=\frac{1}{q^{mn}}\int_{\mathcal{L}_n}\left(\#\left(\Lambda\cap x^m\mathcal{O}_{\infty}^n\right)-1\right)\: dm_{\mathcal{L}_n}(\Lambda).
    \end{split}
    \end{equation}
    Note that if $\Lambda\in \mathcal L_n^{\le {q^{m+1}}}$, then, $\lambda_n(\Lambda)\leq q^{m+1}$, so that 
    \begin{equation}
        \widetilde{f_m}(\Lambda)=\frac{1}{q^{mn}}\prod_{i=1}^n\left\lceil\frac{q^{m+1}}{\lambda_i(\Lambda)}\right\rceil-\frac{1}{q^{mn}}=\frac{q^{(m+1)n}}{q^{mn}\det(\Lambda)}-\frac{1}{q^{mn}}=q^{n}-\frac{1}{q^{mn}}.
    \end{equation}
Observe that for any $\Lambda\in \mathcal L_n^{> q^{m+1}}$, where $m\in \mathbb N$,
\[
\lambda_1(\Lambda) < \frac 1 {q^{(m+1)/(n-1)}}
\quad \text{and hence} \quad
\alpha(\Lambda) > q^{(m+1)/(n-1)}.
\]

Moreover, in the spirit of the proof of Lemma~\ref{Function Field Schmidt Lemma}, one can take a uniform constant $C>0$ for which $\widetilde{f_m}\le C\alpha$ for all $m\in \mathbb N$. It follows that
\[\begin{split}
\int_{\mathcal L_n} \widetilde{f_m} dm_{\mathcal L_n}
&=
\int_{\mathcal L_n^{\le q^{m+1}}} \widetilde{f_m} dm_{\mathcal L_n}+\int_{\mathcal L_n^{> q^{m+1}}} \widetilde{f_m} dm_{\mathcal L_n}\\
&=
\left(q^n-\frac 1{q^{mn}}\right) m_{\mathcal L_n}\left(\mathcal L_n^{\le q^{m+1}}\right)+\int_{\mathcal L_n^{> q^{m+1}}} \widetilde{f_m} dm_{\mathcal L_n}.
\end{split}\]
Thus by Lemma~\ref{lem:alphaMoment},
\[\begin{split}
&\left|\int_{\mathcal L_n} \widetilde{f_m} dm_{\mathcal L_n} - \left(q^n-\frac 1{q^{mn}}\right) m_{\mathcal L_n}\left(\mathcal L_n^{\le q^{m+1}}\right)\right|
\le
C\int_{\mathcal L_n^{> q^{m+1}}} \alpha\: dm_{\mathcal L_n}\\
&\hspace{0.4in}=
C\int_{\mathcal L_n^{> q^{m+1}}} \alpha^{1.5}\cdot \alpha^{-0.5} dm_{\mathcal L_n}
\le \frac {C}{q^{(m+1)/2(n-1)}} \int_{\mathcal L_n^{> q^{m+1}}} \alpha^{1.5} dm_{\mathcal L_n}
\rightarrow 0
\end{split}\]
as $m\rightarrow \infty$. 
By Mahler's compactness criterion \cite[Theorem 1.1]{KST} (see also \cite{Mah}), $\lim_{m\rightarrow \infty}m_{\mathcal L_n}\left(\mathcal L_n^{\le q^{m+1}}\right)=m_{\mathcal L_n}(\mathcal L_n)=1$.
On the other hand,
$$\int_{\mathcal L_n} \widetilde{f_m} dm_{\mathcal L_n} =
c_1\int_{\mathcal{K}_{\infty}^n}f_md\operatorname{Vol}+c_0f_m(0)=\frac{c_1q^{mn}+c_0}{q^{mn}}\rightarrow c_1$$
as $m\rightarrow \infty$. Thus, $c_1=q^n$. 

Now, consider the sequence of test functions $h_m=\mathbf{1}_{x^{-m}\mathcal{O}_{\infty}^n}$, $m\in \mathbb N$, to compute $c_0$. By Lemma \ref{lem:BK}, 
\begin{equation}
\begin{split}
\label{eqn:STSmallBalls}
    \int_{\mathcal{L}_n}\widetilde{h}_mdm_{\mathcal{L}_n}=\int_{\mathcal{L}_n}\left(\#\left(\Lambda\cap x^{-m}\mathcal{O}_{\infty}^n\right)-1\right)dm_{\mathcal{L}_n}(\Lambda).
\end{split}
\end{equation}
Note that $\lambda_1(\Lambda)\geq q^{-(m-1)}$ if and only if
\[
\#\left(\Lambda\cap x^{-m}\mathcal{O}_{\infty}^n\right)=1.
\]
Hence,
\[\begin{split}
0\le
\int_{\mathcal L_n} \widetilde{h_m} dm_{\mathcal L_m}
&=\int_{\{\Lambda\in \mathcal L_n: \alpha_1(\Lambda)\ge q^{(m-1)}\}} \widetilde{h_m} dm_{\mathcal L_m}\\
&\le C \int_{\{\Lambda\in \mathcal L_n: \alpha_1(\Lambda)\ge q^{(m-1)}\}} \alpha\:dm_{\mathcal L_m}\\
&\le \frac {C} {q^{(m-1)/2}} \int_{\{\Lambda\in \mathcal L_n: \alpha_1(\Lambda)\ge q^{(m-1)}\}} \alpha^{1.5} dm_{\mathcal L_n},
\end{split}\]
where $C>0$ is the same constant for functions $f_m$ as before. By Lemma~\ref{lem:alphaMoment} again, the above quantity converges to zero as $m$  goes to $\infty$.

On the other hand,
\begin{equation}
\int_{\mathcal L_n} \widetilde{h_m} \:dm_{\mathcal L_m}=    q^n\int_{\mathcal{K}_{\infty}^n}h_md\operatorname{Vol}+c_0h_m(0)=q^{-(m-1)n}+c_0\rightarrow c_0
\end{equation}
as $m\rightarrow \infty$. Thus $c_0=0$ and one obtains the integral formula \eqref{eqn:SiegelFormula}.
\end{proof}
\subsection{Higher Moments}
We now compute the $k$-th moment of the Siegel transform
\begin{equation}
\begin{split}
\label{eqn:kMoment}
    \int_{\mathcal{L}_n}\widetilde{f}\:^kdm_{\mathcal{L}_n}
    &=\int_{\mathcal{L}_n}\left(\sum_{\mathbf{v}\in \mathcal{R}^n\setminus\{0\}}f(g\mathbf{v})\right)^kdm_{\mathcal{L}_n}(g\mathcal{R}^n)\\
    &=\int_{\mathcal{L}_n}\sum_{\mathbf{v}_1,\dots,\mathbf{v}_k\in \mathcal{R}^n\setminus\{0\}}f(g\mathbf{v}_1)\cdots f(g\mathbf{v}_k)dm_{\mathcal{L}_n}(g\mathcal{R}^n).
\end{split}
\end{equation}
If $F:(\mathcal K_\infty^n)^k\rightarrow \mathbb C$ is a smooth bounded function, and we define 
\[
\widetilde{F}(\Lambda)=\sum_{\mathbf{v}_1,\dots,\mathbf{v}_k\in \mathcal{R}^n\setminus\{0\}} F(g\mathbf{v}_1, \ldots, g\mathbf{v}_k),
\]
then the $k$-th moment formula of $f$ is the special case of the integral formula for $\widetilde{F}$, when $F=f\otimes \cdots \otimes f$.
We want to apply Riesz--Kakutani--Markov theorem again to deduce the higher moment formula and for this, we first divide $(\mathcal{R}^n)^k$ into the equivalence classes with respect to the relation
\[\begin{gathered}
(\mathbf{v}_1, \ldots, \mathbf{v}_k)\sim (\mathbf{w}_1, \ldots, \mathbf{w}_k)\in (\mathcal R^n)^k\quad\text{if}\\
(\mathbf{v}_1, \ldots, \mathbf{v}_k)=(g\mathbf{w}_1, \ldots, g\mathbf{w}_k)\quad\text{for some }g\in \SL_n(\mathcal K_\infty).
\end{gathered}\]
As in the real case \cite{Rogers1955} or $S$-arithmetic case \cite{Han}, 
such equivalence classes correspond to rational subspaces in $\mathcal K_\infty^k$, which can be represented by \emph{admissible matrices}, which are defined below.

Let $\Rmo$ be the set of all monic polynomials, that is, the set of $f\in \mathcal R$ whose leading coefficient is one. We define the \emph{great common divisor} $D:=\gcd(f_1,\dots,f_d)\in \Rmo$ for $f_1,\dots,f_d\in \mathcal{R}\setminus\{0\}$ as the unique monic polynomial satisfying the following properties:
\begin{itemize}
    \item $D\mid f_i$ for every $i=1,\dots,d$;
    \item if  $D'\in \mathcal R$ is such that $D'\mid f_i$ for all $i$, then $D\mid D'$.
\end{itemize}

\begin{definition}[Admissible Matrices]\label{def: admissible matrix}
Let $1\leq r\leq k$.
\begin{enumerate}
    \item We say that a matrix $D\in M_{r\times k}(\mathcal{R})$ is \emph{reduced} if $\gcd(\{D_{ij} : 1\leq i\leq r,1\leq j\leq k\})=1$. 
    \item For $Q\in \Rmo$, let $\mathfrak{D}^k_{r,Q}$ denote the set of reduced matrices $D\in M_{r\times k}(\mathcal R)$ such that there exist $1\leq i_1<\dots<i_r\leq k$ such that
    \begin{enumerate}
        \item no column of $D$ is the zero vector,
        \item $([D]^{i_1},\dots,[D]^{i_r})=Q\mathrm{Id}_r$, and
        \item $D_{\ell j}=0$ for every $1\leq \ell\leq r, 1\leq j<i_\ell$,
    \end{enumerate}
    where $[D]^j$ is the $j$-th column of $D$. 
    \item Let $N(D,Q)$ denote the number of row vectors $\mathbf{v}\in \mathcal{R}^{r}_{\le \deg(Q)}$, such that $\frac{1}{Q}\mathbf{v}D\in \mathcal{R}^k$.
\end{enumerate} 
\end{definition}

In the real case, $N(D,q)$ is defined as the number of row vectors in $\{0,1,\ldots, q-1\}^r$, which is a fundamental domain for $\mathbb Z^r/q\mathbb Z^r$. Thus, $\frac 1 q \mathbf{v} D\in \mathbb Z_S^k$, and $N(D,q)$ is related to the covolume of the lattices in $(\mathbb R^n)^k$ which are associated to the matrix $\frac D q$.
In the function field case, since $\operatorname{covol}(\mathcal R^k)=q^{k}$, to avoid a redundant $q^r$-term in the moment formula, we restrict the definition of $N(D,Q)$ to vectors in $\mathcal R^r_{\le \deg(Q)}$, which is a fundamental domain for $(x\mathcal R/Q\mathcal R)^r$. In particular, $N(\mathrm{Id}_k, 1)=q^{k}$ and moreover, $\left| \mathcal{R}^r_{\leq \deg(Q)}\right|=q^r\left| \mathcal{R}_{<\deg(Q)}^r\right|$.
\begin{theorem}\label{thm:kMomentsForm} 
Let $F:(\mathcal K_\infty^n)^k\rightarrow \mathbb C$ be a bounded and compactly supported function. It holds that
        \begin{equation}\begin{split}
    \label{eqn:Rogersk}
&\int_{\mathcal{L}_n} \widetilde{F} (g\mathcal R^n) dm_{\mathcal{L}_n}(g\mathcal R^n) 
=
q^{nk}\int_{(\mathcal{K}_{\infty}^n)^k} F(\mathbf v_1, \ldots, \mathbf v_n) d\mathbf v_1 \cdots d\mathbf v_n\\
&\hspace{1in}
+\sum_{r=1}^{k-1}\sum_{Q\in \Rmo}\sum_{D\in \mathfrak{D}^k_{r,Q}}
\frac{N(D,Q)^n}{\vert Q\vert^{nr}}\int_{\left(\mathcal{K}_{\infty}^n\right)^r}F\left(\frac{1}{Q}(\mathbf{v}_1,\dots,\mathbf{v}_r)D\right)d\mathbf{v}_1\dots d\mathbf{v}_r.\end{split}\end{equation}
\end{theorem}
Let $(\mathbf v_1, \ldots, \mathbf v_k)\in ( \mathcal{R}\setminus\{0\})^k$ be such that $\dim \operatorname{span}\{\mathbf v_1, \ldots, \mathbf v_k\}=r$, where $r=1, \ldots, k$. One can find a unique set $\{1\le i_1 < \cdots < i_r\le k\}$ of indices such that 
\begin{center}
$\mathbf{v}_j\in \operatorname{span}\{\mathbf{v}_{i_1},\dots,\mathbf{v}_{i_\ell}\}$, whenever $i_{\ell}\leq j<i_{\ell+1}$. 
\end{center}
In particular, the vectors $\mathbf v_{i_1}, \ldots, \mathbf v_{i_r}$ are linearly independent.

By setting $\mathbf{w}_1:=\mathbf{v}_{i_1},\dots,\mathbf{w}_r:=\mathbf{v}_{i_r}$, we obtain a matrix $C\in M_{r\times k}(\mathcal{K})$ such that 
\begin{equation}
    (\mathbf{w}_1,\dots,\mathbf{w}_r)C=(\mathbf{v}_1,\dots,\mathbf{v}_k).
\end{equation}
Since $\mathbf{v}_j$ is spanned by $\{\mathbf{v}_{i_1},\dots,\mathbf{v}_{i_{\ell}}\}$ whenever $i_{\ell}\leq j<i_{\ell+1}$, for every $1\leq \ell\leq r$ and every $1\leq j<i_{\ell}$, we have $C_{\ell,j}=0$. Let $Q$ be the unique polynomial in $\Rmo$ such that
$$D:=QC\in M_{r\times k}(\mathcal{R})$$ is reduced.

Clearly, $([D]^{i_1},\dots,[D]^{i_r})=Q\mathrm{Id}_r$ and it follows that $D\in\mathfrak{D}^k_{r,Q}$. Denote
\begin{equation}\label{Def: Phi}
\Phi(D,Q)=\left\{(\mathbf{w}_1,\dots,\mathbf{w}_r)\in (\mathcal R^n)^r:\begin{matrix}\frac{1}{Q}(\mathbf{w}_1,\dots,\mathbf{w}_r)D\in(\mathcal{R}^n)^k\;\text{and}\\ 
\operatorname{rk}\{\mathbf{w}_1,\dots,\mathbf{w}_r\}=r\end{matrix}\right\}.
\end{equation}

Note that if $r=k$, only possible $Q$ is $1$ and $\mathfrak{D}^k_{k,1}=\{\mathrm{Id}_k\}$. 

In general, there is a partition of $(\mathcal{R}\setminus\{0\})^k$ with respect to $\mathfrak{D}^k_{r,Q}$ and using $\Phi(D,Q)$ as follows.
\begin{equation}
\begin{split}
\label{eqn:(R^n)^kSplit}
    \{(\mathbf{v}_1,\dots,\mathbf{v}_k):\mathbf{v}_i\in \mathcal{R}^n\setminus \{0\}\}= \left\{(\mathbf{v}_1,\dots,\mathbf{v}_k):\operatorname{rk}\{\mathbf{v}_1,\dots,\mathbf{v}_k\}=k\right\}\\
    \sqcup\bigsqcup_{r=1}^{k-1}\bigsqcup_{Q\in \Rmo}\bigsqcup_{D\in \mathfrak{D}^k_{r,Q}}\left\{\frac{1}{Q}(\mathbf{w}_1,\dots,\mathbf{w}_r)D:(\mathbf{w}_1,\dots,\mathbf{w}_d)\in \Phi(D,Q)\right\}.
\end{split}
\end{equation}
By \eqref{eqn:(R^n)^kSplit} and Tonelli's integration theorem, 
\begin{equation}
    \begin{split}
    &\int_{\mathcal{L}_n}\widetilde{F}(g\mathcal R^n) dm_{\mathcal{L}_n}
    \int_{\mathcal{L}_n} \hspace{-0.05in}\sum_{\scriptsize \begin{matrix}
        \mathbf{v}_1,\dots,\mathbf{v}_k\in \mathcal{R}^n\\
        \operatorname{rk}\{\mathbf{v}_1,\dots,\mathbf{v}_k\}=k
    \end{matrix}}F(g\mathbf{v}_1,\dots,g\mathbf{v}_k)dm_{\mathcal{L}_n}\\
    &\hspace{0.8in}+\sum_{r=1}^{k-1}\sum_{Q\in\Rmo}\sum_{D\in \mathfrak{D}^k_{r,Q}}
    \int_{\mathcal{L}_n}\sum_{\scriptsize \begin{array}{c}
    (\mathbf w_1, \ldots, \mathbf w_r)\\
    \text{in }\Phi(D,Q)\end{array}}
    F\left(\frac{1}{Q}g\mathbf{w}_1D,\dots,\frac{1}{Q}g\mathbf{w}_rD\right)dm_{\mathcal{L}_n}.
    \end{split}
\end{equation}

Thus, Theorem \ref{thm:kMomentsForm} is deduced from the following theorem on the summands with $r=1,\dots,k-1$ from \eqref{eqn:Rogersk}. 
\begin{theorem}
\label{thm:rPart}
    For $1\leq k\leq n-1$ and $1\leq r\leq k$, let $Q\in \Rmo$ and let $D\in \mathfrak{D}^k_{r,Q}$, where $\mathfrak{D}^k_{r,Q}$ is defined as in Definition~\ref{def: admissible matrix}. Then we have
    \begin{equation}\label{eq: rPart}
    \begin{split}
        &\int_{\mathcal{L}_n}\sum_{\scriptsize \begin{array}{c} (\mathbf{w}_1,\dots,\mathbf{w}_r)\\
        \text{in }\Phi(D,Q)\end{array}}F\left(\frac{1}{Q}(g\mathbf{w}_1,\dots,g\mathbf{w}_r)D\right)dm_{\mathcal{L}_n}\\
        &\hspace{1in}=\frac{N(D,Q)^n}{\vert Q\vert^{nr}}\int_{\left(\mathcal{K}_{\infty}^n\right)^r}F\left(\frac{1}{Q}(\mathbf{v}_1,\dots,\mathbf{v}_r)D\right)d\mathbf{v}_1\dots d\mathbf{v}_r.
    \end{split}
    \end{equation}
\end{theorem}
To prove Theorem \ref{thm:rPart}, define
    \begin{equation}\label{Def: Psi}\Psi(D,Q)=\left\{(\mathbf{w}_1,\dots,\mathbf{w}_r)\in \left(\mathcal{R}^n\right)^r:\frac{1}{Q}(\mathbf{w}_1,\dots,\mathbf{w}_r)D\in \left(\mathcal{R}^n\right)^k\right\},\end{equation}
    so that $\Phi(D,Q)=\Psi(D,Q)\setminus E$, where
        $$E=\{(\mathbf{w}_1,\dots,\mathbf{w}_r)\in(\mathcal{R}^n)^r:\operatorname{rk}_{\mathcal{K}}\{\mathbf{w}_1,\dots,\mathbf{w}_r\}<r\}.$$
        
\begin{lemma}
\label{lem:Psi(D,Q)Covol}
    For every $Q\in \Rmo$ and for every $D\in \mathfrak{D}^k_{r,Q}$, $\Psi(D,Q)$ is a lattice in $(\mathcal{R}^n)^r$ with covolume $|Q|^{nr}/N(D,Q)^n$. 
\end{lemma}
\begin{proof}
    Firstly, it is easy to verify that $\Psi(D,Q)$ is an additive group, satisfying $$Q(\mathcal{R}^n)^r\subseteq \Psi(D,Q)\subseteq (\mathcal{R}^n)^r.$$ 
    Hence, $\Psi(D,Q)$ is a lattice in $(\mathcal{R}^n)^r$. By taking the projection $\pi_Q: (\mathcal{R}^n)^r\rightarrow(\mathcal{R}^n/Q\mathcal{R}^n)^r$, one obtains that 
    \begin{equation*}
        \operatorname{covol}(\Psi(D,Q))=\frac{\#(\mathcal{R}^n/Q\mathcal{R}^n)^r}{\#\left\{(\mathbf{v}_1,\dots,\mathbf{v}_r)\in (\mathcal{R}^n/Q\mathcal{R}^n)^r:\frac{1}{Q}(\mathbf{v}_1,\dots,\mathbf{v}_r)D\in (\mathcal{R}^n)^k\right\}}=\frac{\vert Q\vert^{nr}}{N(D,Q)^n}.
    \end{equation*}
\end{proof}

\begin{proof}[Proof of Theorem \ref{thm:rPart}]
 For $1\le r \le k$ and $Q\in \Rmo$, fix $D\in \mathfrak D^k_{r,Q}$, and let $\Phi(D,Q)$ and $\Psi(D,Q)$ be as in \eqref{Def: Phi} and \eqref{Def: Psi} respectively.

For a compactly supported, continuous, and bounded function $H:\left(\mathcal{K}_{\infty}^n\right)^r\rightarrow \mathbb{R}_{\geq 0}$, define
\[\begin{split}
\phi_H(g\Gamma)&=\sum_{(\mathbf{w_1},\dots,\mathbf{w}_r)\in \Phi(D,Q)}H\left(g\mathbf{w}_1,\dots,g\mathbf{w}_r\right),\quad \forall g\Gamma\in \mathcal L_n.\\
\end{split}\]
Then equation \eqref{eq: rPart} follows from the following formula
\begin{equation}\label{eq: H-integral}
\int_{\mathcal L_n} \phi_H(g\Gamma) dm_{\mathcal L_n}(g\Gamma)
=\frac {N(D,Q)^r} {|Q|^{nr}} \int_{(\mathcal K^n)^r} H(\mathbf w_1, \ldots, \mathbf w_r) d\mathbf w_1 \cdots d\mathbf w_r,
\end{equation}
when taking $H=H_F$ defined by
    \[
    H_F(\mathbf w_1, \ldots, \mathbf w_r)
    :=F\left(\frac 1 Q (\mathbf w_1, \ldots, \mathbf w_r)D\right).
    \]
    
Let $E=\{(\mathbf w_1, \ldots, \mathbf w_r)\in \mathcal K_\infty^r: \rk(\mathbf w_1, \ldots, \mathbf w_r)\le r-1\}$. Note that $(\mathcal K^n_\infty)^r\setminus E$ is locally compact Hausdorff space and consists of a single orbit of the $\mathrm{G}$-action defined by $g.(\mathbf w_1, \ldots, \mathbf w_r)=(g\mathbf w_1, \ldots, g\mathbf w_r)$ for all $g\in \mathrm{G}$ and $(\mathbf w_1, \ldots, \mathbf w_r)\in (\mathcal K^n_\infty)^r$.
Moreover, $E$ is locally diffeomorphic to $K^{(r-1)(n+1)}_\infty$. Since $r\le n-1$, we have $(r-1)(n+1)<nr$, and the unique $\mathrm{G}$-invariant measure $m_{(\mathcal K_\infty^n)^r\setminus E}:=\left.m_{(\mathcal K_\infty^n)^r}\right\vert_{(\mathcal K_\infty^n)^r\setminus E}$ on $(\mathcal K_\infty^n)^r\setminus E$ is measure-theoretically equal to the Haar measure $m_{(\mathcal K_\infty^n)^r}$.

First, we claim that for $H\in C_c\big((\mathcal K_\infty^n)^r\setminus E\big)$, the integral $\int_{\mathcal L_n} \phi_H(g\Gamma) dm_{\mathcal L_n}$ is finite: one can take a compactly supported, bounded, continuous function $h:\mathcal{K}_{\infty}^n\rightarrow \mathbb{R}_{\geq 0}$, satisfying $H(\mathbf{v}_1,\dots,\mathbf{v}_r)\leq h(\mathbf{v}_1)\cdots h(\mathbf{v}_r)=:h^{\otimes r}(\mathbf{v}_1,\dots,\mathbf{v}_r)$. As a consequence, $\phi_H(\Lambda)\leq \big(\widetilde{h}(\Lambda)\big)^r$ for every $\Lambda\in \mathcal{L}_n$. Hence, by Theorem \ref{lem:alphaMoment} and Lemma \ref{cor:LattPointCnt} \eqref{lem:SchmidtLem}, there exists $C_h\in \mathbb R_{>0}$, such that
    \begin{equation*}
        \int_{\mathcal{L}_n} \phi_H \: dm_{\mathcal{L}_n}\leq \int_{\mathcal{L}_n}\widetilde{h}^r\:dm_{\mathcal{L}_n}\leq C_h\int_{\mathcal{L}_n}\alpha^rdm_{\mathcal{L}_n}<\infty.
    \end{equation*}
Thus by Riesz--Kakutani--Markov theorem, together with the fact that the positive linear functional $$H\in C_c\big((\mathcal K_\infty^n)^r\setminus E\big)\mapsto \int_{\mathcal L_n} \phi_H \:dm_{\mathcal L_n}$$ is $\mathrm{G}$-invariant, there exists $c>0$ such that for any $H\in C_c\big((\mathcal K_\infty^n)^r\setminus E\big)$,
    \begin{equation}
    \label{eqn:inthat(F)=cint(F)}    \int_{\mathcal{L}_n}\phi_H\: dm_{\mathcal{L}_n}=c\int_{\left(\mathcal{K}_{\infty}^n\right)^r} H\: dm_{\left(\mathcal{K}_{\infty}^n\right)^r}.\end{equation}

    \vspace{0.1in}
    Second, let us compute the constant $c>0$ for the functions in $C_c\left((\mathcal K^n_\infty)^r\setminus E\right)$. For each $t\in \mathbb N$, consider $B_t=B_{(\mathcal{K}_{\infty}^n)^r}(0,q^t)\setminus N_{1}(E)\subseteq (\mathcal{K}_{\infty}^n)^r$, where $N_{1}(E)$ is the $1$-neighborhood of $E$. 
    Define a function $H_t$ on $(\mathcal K^n_\infty)^r\setminus E$ as 
    $H_t=\frac 1 {q^{trn}} \mathbf 1_{B_t}.$
    Observe that
    \[\begin{gathered}
    m_{(\mathcal K_\infty^n)^r} \left(B_{(\mathcal{K}_{\infty}^n)^r}(0,q^t)\cap  N_{1}(E)\right)=O\left(q^{t(r-1)(n+1)}\right),\\
    \#\left(\Lambda\cap B_{(\mathcal{K}_{\infty}^n)^r}(0,q^t)\cap  N_{1}(E)\right)=O_\Lambda\left(q^{t(r-1)(n+1)}\right)
    \end{gathered}\] 
    for any lattice $\Lambda\in \mathcal{L}_n$. More precisely, the implied constant in the second Big-O depends on the successive minima $\lambda_1(\Lambda),\cdots ,\lambda_n(\Lambda)$ of the lattice $\Lambda$.
    It follows that
    \begin{equation}\label{eqn:H_tInt}
    \lim_{t\rightarrow \infty}\int_{(\mathcal{K}_{\infty}^n)^r} H_t\:dm_{(\mathcal{K}_{\infty}^n)^r} =\lim_{t\rightarrow \infty}\int_{(\mathcal{K}_{\infty}^n)^r}\frac 1 {q^{trn}}\mathbf 1_{B_{(\mathcal{K}_{\infty}^n)^r}(0,q^t)}\:dm_{(\mathcal{K}_{\infty}^n)^r}=1.\end{equation}
    Moreover, for every $\Lambda \in \mathcal{L}_n$,
    \begin{equation}\label{eqn:phi_H_t}\begin{split}
    \phi_{H_t}(\Lambda)
    &=\frac 1 {q^{nrt}} \#\left(g\Phi(D,Q)\cap B_t\right)
    =\frac 1 {q^{nrt}} \#\left(g\Psi(D,Q)\cap B_t\right)\\
    &=\frac 1 {q^{nrt}} \#\left(g\Psi(D,Q) \cap (B_{(\mathcal K_\infty^n)^r}(0, q^t)\setminus N_{1}(E))\right)\\
    &\underset{t\rightarrow \infty}{\longrightarrow} \frac 1 {q^{nrt}} \#(g\Psi(D,Q)\cap B_{(\mathcal K_\infty^n)^r}(0, q^t)),
    \end{split}\end{equation}
    where the limit is uniform on any compact subset of $\mathcal L_n$. By Lemma \ref{lem:BK} and Lemma \ref{lem:Psi(D,Q)Covol}, the right hand side of \eqref{eqn:phi_H_t} converges to $N(D,Q)^n/|Q|^{nr}$ as $t\rightarrow \infty$, uniformly on any compact subset of $\mathcal L_n$.
    Thus, by \eqref{eqn:inthat(F)=cint(F)} and \eqref{eqn:H_tInt}, we have
    \[
    c=\lim_{t\rightarrow \infty} c\int_{\left(\mathcal{K}_{\infty}^n\right)^r} H_t\: dm_{\left(\mathcal{K}_{\infty}^n\right)^r}=\lim_{t\rightarrow \infty} \int_{\mathcal L_n} \phi_{H_t}\: dm_{\mathcal L_n}=\frac {N(D,Q)^n} {|Q|^{nr}}.
    \]

\vspace{0.1in}
    Finally, we show that the identity \eqref{eqn:inthat(F)=cint(F)} extends to functions in $C_c\left((\mathcal K^n_\infty)^r\right)$. 
For any given $F\in C_c\left((\mathcal K^n_\infty)^{r}\right)$, consider the sequence of functions $F_t \in C_c\left((\mathcal K^n_\infty)^{r}\setminus E\right)$ defined by
\[
F_t=F\mid^{}_{(\mathcal K_\infty^n)^r\setminus N_{\frac 1 {q^{-t}}}(E)},
\]
where $N_{\frac 1 {q^{-t}}}(E)$ is the $\frac 1 {q^{-t}}$-neighborhood of $E$.
Note that $(F_t)\nearrow F$, and thus there exists a constant $C=C_F>0$ such that
\begin{equation}\label{eqn:phi<alpha^r}
\phi_{F_t}\le \phi_F\le C\alpha^r, \quad\forall t\in \mathbb N.\end{equation}
Since $m_{(\mathcal K_\infty^n)^r} (\supp F \cap N_{\frac 1 {q^{-t}}}(E))=O_F(q^{-t(r-1)(n+1)})$, it follows that
\[
\lim_{t\rightarrow \infty} \int_{(\mathcal K_\infty^n)^r} F_t \: dm_{(\mathcal K_\infty^n)^r}=\int_{(\mathcal K_\infty^n)^r} F\: dm_{(\mathcal K_\infty^n)^r}.
\]
Hence, it suffices to show that \begin{equation}\label{eqn:IntConverge}\lim_{t\rightarrow \infty}\int_{\mathcal L_n} \phi_{F_t} dm_{\mathcal L_n}=\int_{\mathcal L_n} \phi_{F} dm_{\mathcal L_n}.\end{equation}
Since $F$ is compactly supported, for any $\Lambda=g\mathcal R^n\in \mathcal L_n$, there is $t_0>1$ depending on $\|g\|,\;\|g^{-1}\|$ (thus uniform on any compact set on $\mathcal L_n$) such that if $t\ge t_0$, we have $\phi_{F_t}=\phi_F$, that is, $\supp F \cap N_{\frac 1 {q^{-t}}}(E) \cap g\Phi(D,Q)=\emptyset$. 

Therefore, for any $\varepsilon>0$, there exists $\ell=\ell_{\varepsilon}\in \mathbb N$ for which $\int_{\mathcal L_n^{> q^{\ell}}} \alpha^{r+\frac 1 2} dm_{\mathcal L_n}<\frac \varepsilon {C}$. Hence, by \eqref{eqn:phi<alpha^r} and \eqref{lem:UM}, for all sufficiently large $t$, we have
\[
\left|\int_{\mathcal L_n} \phi_F \: dm_{\mathcal L_n} -\int_{\mathcal L_n} \phi_{F_t} \: dm_{\mathcal L_n}\right|\leq \lim_{t\rightarrow \infty} \int_{\mathcal L_n^{> q^{\ell}}} \left|\phi_F- \phi_{F_t}\right| dm_{\mathcal L_n} <\frac{\varepsilon}{C}\cdot C=\varepsilon.
\]
Thus, \eqref{eqn:IntConverge} holds, so that all functions $F\in C_c((\mathcal{K}_{\infty}^n)^r)$ satisfy \eqref{eqn:inthat(F)=cint(F)}. 
\end{proof}

\subsection{Upper Bound of the Variance}
Let us first restate the special case of Theorem~\ref{thm:kMomentsForm} when $k=2$.
\begin{lemma}\label{lem:Rogers2nd} Let $n\ge 3$ and $F:(\mathcal K_\infty^n)^2\rightarrow \mathbb C$ be a smooth function. Then
\begin{equation}\label{Eqn: second moment}\begin{split}
    \int_{\mathcal L_n} \widetilde{F} dm_{\mathcal L_n}
    =&q^{2n} \int_{(\mathcal K_\infty^n)^2} F \:d\vol_{(\mathcal K_\infty^n)^2}\\
    &+\sum_{Q\in \Rmo}
    \sum_{\scriptsize \begin{array}{c}
    a\in \mathcal R\setminus \{0\}\\
    \gcd(a,Q)=1
    \end{array}} \frac {q^n}{|Q|^n} \int_{\mathcal K_\infty^n} F\left(\mathbf v, \frac {a} {Q}\mathbf v\right) d\vol(\mathbf v).
\end{split}\end{equation}
\end{lemma}
\begin{proof}
For each $Q\in \Rmo$, the matrix $D\in \mathfrak D^2_{1,Q}$ is of the form $\begin{pmatrix} Q & a\end{pmatrix}$, where $a\in \mathcal R\setminus \{0\}$ is coprime with $Q$.
In view of Theorem \ref{thm:kMomentsForm}, it suffices to compute $N(D,Q)$. Note that a vector $\mathbf{v}\in \mathcal{K}_{\infty}$ satisfies $\left(\mathbf{v},\frac{a}{Q}\mathbf{v}\right)=\frac{1}{Q}\mathbf{v}\begin{pmatrix}
    Q&a
\end{pmatrix}\in \mathcal{R}^2$ if $\mathbf{v}\in \mathcal{R}$ and $a\mathbf{v}\in Q\mathcal{R}$. Since $\gcd(a,Q)=1$, then, $\mathbf{v}\in Q\mathcal{R}$. Hence, 
\begin{equation}
\begin{split}
    N(D,Q)=\#Q\mathcal{R} \cap \mathcal{R}_{\leq \deg(Q)}=\#(Q\mathbb{F}_q)=q.
\end{split}
\end{equation}
Thus, Equation \eqref{Eqn: second moment} follows directly from Theorem~\ref{thm:kMomentsForm}.
\end{proof}

From the Lemma \ref{lem:Rogers2nd}, together with some basic number theoretic properties on function fields, one can deduce Theorem \ref{thm:Rogers}. 

\begin{proof}[Proof of Theorem \ref{thm:Rogers}]
By Lemma \ref{lem:Rogers2nd}, it suffices to prove that for a measurable set $B\subseteq \mathcal K_\infty^n$,
\begin{equation}\label{eqn:var}\begin{split}
\operatorname{Var}(X_{B})
=\sum_{Q\in \Rmo}
    \sum_{\scriptsize \begin{array}{c}
    a\in \mathcal R\setminus \{0\}\\
    \gcd(a,Q)=1
    \end{array}} \frac {q^n}{|Q|^n} \int_{\mathcal K_\infty^n} X_B\left(\mathbf v\right)X_B\left(\frac {a} {Q}\mathbf v\right) d\vol(\mathbf v)
\end{split}\end{equation}
is bounded by $C_n\vol(B)$, where $C_n$ is the constant given as in \eqref{eqn:C_n}.
Note that 
\begin{equation}\label{eqn:VolBND}\begin{split}
    \int_{\mathcal{K}_{\infty}}X_B(\mathbf{v})X_B\left(\frac{a}{Q}\mathbf{v}\right)d\operatorname{Vol}(\mathbf{v})
    &=\operatorname{Vol}\left(B\cap \frac{Q}{a}B\right)\\
    &\hspace{-0.7in}\leq \begin{cases}
    \operatorname{Vol}(B)& \text{if }\deg(Q)\ge \deg(a);\\[0.05in]
    \left|\frac{Q}{a}\right|^n\operatorname{Vol}(B)& \text{if }\deg(Q)<\deg(a).
\end{cases}\end{split}\end{equation}
Thus, one can consider three cases which are $\deg(Q)\ge \deg(a)$, $\deg(Q)=\deg(a)$, and $\deg(Q)<\deg(a)$. For the last case, one can switch the roles of $Q$ and $a$ by the following equality that holds for any integrable function $f:\mathcal K_\infty^n \rightarrow \mathbb C$:
\begin{equation}\label{eqn:symmetry}\begin{split}
\frac 1 {|Q|^n} \int_{\mathcal K^n_\infty} f\left(\mathbf v\right)f\left( \frac a Q \mathbf v\right) d\vol(\mathbf v)
&=\int_{\mathcal K^n_\infty} f\left(Q\mathbf v\right)f\left( a\mathbf v \right) d\vol(\mathbf v)\\
&=\frac 1 {|a|^n} \int_{\mathcal K^n_\infty} f\left(\frac Q a \mathbf v\right)f\left(\mathbf v\right)d\operatorname{Vol}(\mathbf{v}).
\end{split}\end{equation}
Moreover, since $\{a\in \mathcal R: \deg(a)=d\}=\mathbb F_q^{\times} \times \{a\in \Rmo: \deg(a)=d\}$ ($d\neq 0$), by changing the order of summation of $Q$ and $a$ in \eqref{eqn:var}, it follows that \eqref{eqn:var} is bounded above by
\begin{equation}
\label{eqn:sumVOlBND}
    q^n\sum_{d=0}^{\infty}\frac{1}{q^{nd}}
    \sum_{\scriptsize \begin{array}{c}
    Q\in \Rmo\\
    \deg(Q)=d\end{array}}\left(q\times \hspace{-0.2in}\sum_{\scriptsize \begin{array}{c}
    a\in \mathcal{R}_{< d}\\
    \gcd(a,Q)=1\end{array}}\hspace{-0.15in}\operatorname{Vol}(B)+\sum_{\scriptsize\begin{matrix}
        a\in \mathcal R_{=d}\\
        \gcd(a,Q)=1
    \end{matrix}}
    \operatorname{Vol}(B)\right).
\end{equation}
We now bound the first internal summand in \eqref{eqn:sumVOlBND}. Recall the Euler $\varphi$-function for a function field defined as
$$\varphi(Q)=\#\{a\in \mathcal{R}_{<\deg(Q)}: \gcd(a,Q)=1\},$$ 
and that by \cite[Proposition 2.7]{Ros}, we have 
$$\sum_{Q\in \mathcal{R}_{\text{monic}}\cap \mathcal{R}_{=d}}\varphi(Q)=q^{2d}(1-q^{-1}).$$ 
To count the number of $a\in \mathcal R$ with $\vert a\vert=q^m$, where $m\ge d$ (we only need the case when $m=d$) and $a$ is coprime with $Q$, we use the inclusion-exclusion principle. Let $P_1, \cdots, P_{\ell}$ be irreducible factors of $Q$. It holds that
\begin{equation}
\begin{split}
\label{eqn:CoprimeGenDeg}
   &\#\{\vert a\vert=q^m:\gcd(a,Q)=1\}
   =\#\mathcal{R}_{=m}\setminus \bigcup_{i=1}^{\ell}P_i\mathcal{R}\\
   &\hspace{0.2in}=q^m(q-1)+\sum_{k=1}^{\ell}(-1)^k\sum_{1\leq i_1<\dots<i_k\leq \ell}\#\mathcal{R}_{=m}\cap P_{i_1}\cdots P_{i_k}\mathcal{R}\\
   &\hspace{0.2in}=q^m(q-1)\left[1+\sum_{k=1}^{\ell}(-1)^k\sum_{1\leq i_1<\cdots <i_k\leq \ell}\frac{1}{\vert P_{i_1}\cdots P_{i_k}\vert}\right]\\
   &\hspace{0.2in}=q^m(q-1)\prod_{ \scriptsize \begin{array}{c} P|Q\\ P\text{ irred.}\end{array}}\left(1-\frac{1}{\vert P\vert}\right)
   =q^m(q-1)\frac{\varphi(Q)}{\vert Q\vert},
\end{split}
\end{equation}
where the last line is obtained through \cite[Proposition 1.7]{Ros}. 
Thus, the fact that $n\geq 3$ implies that
\begin{equation}
\begin{split}
\label{eqn:smallA}
\eqref{eqn:sumVOlBND} =q^n(2q-1)(1-q^{-1})\sum_{d=0}^{\infty}\frac{\operatorname{Vol}(B)}{q^{nd}}q^{2d}=(2q-1)(1-q^{-1})\frac{q^{2n-2}}{q^{n-2}-1}\operatorname{Vol}(B).
\end{split}
\end{equation}
\end{proof}
When the set $B\subseteq \mathcal K_\infty^n$ is a metric ball $B(\alpha, r)$, it is possible to optimize the inequality \eqref{eqn:VolBND}. Together with the lemma below, which is well-known for ultrametic spaces, one can obtain the exact value of the variance $\operatorname{Var}(X_{B(\alpha, r)})$. 
\begin{lemma}
\label{lem:UMBalls}
    Let $\alpha_1,\alpha_2\in \mathcal{K}_{\infty}^n$ and let $r_1,r_2\geq 0$. Then 
    $$B(\alpha_1,r_1)\cap B(\alpha_2,r_2)=\begin{cases}
        \emptyset, &\text{ if }\Vert \alpha_1-\alpha_2\Vert>\max\{r_1,r_2\},\\
        B(\alpha_i,\min\{r_1,r_2\})&\text{ if }\Vert \alpha_1-\alpha_2\Vert\leq \max\{r_1,r_2\}.
    \end{cases}$$
\end{lemma}

\begin{proof}[Proof of Proposition~\ref{lem:BallErr}]
When $\|\alpha\|\le r$, by Lemma \ref{lem:UMBalls}, we have $B(\alpha, r)=B(0, r)$. Thus we may assume that $\alpha$ is the origin. In this case, every intersection $B(0,r)\cap B\left(0, \left|\frac Q a\right| r\right)$ is non-empty. Moreover, the inequality in \eqref{eqn:VolBND} is an equality for every pair of coprime polynomials $(a,Q)$. Thus, by following the proof of Theorem~\ref{thm:Rogers}, we have
$$\operatorname{Var}(X_{B(\alpha, r)})
=\operatorname{Var}(X_{B(0, r)})
=C_n.$$
\vspace{0.1in}
Now, assume that $\|\alpha\|>r$. By Lemma \ref{lem:UMBalls},  $B(\alpha,r)\cap \frac{Q}{a}B(\alpha,r)=B(\alpha,r)\cap B\left(\frac{Q}{a}\alpha,\left|\frac{Q}{a}\right|r\right)\neq \emptyset$ if and only if $\left|1-\frac{Q}{\alpha}\right|\cdot \Vert \alpha\Vert\leq \max\left\{1,\left|\frac{Q}{a}\right|\right\}r$, i.e.,
\begin{equation}\label{eqn:admissRange}
B(\alpha,r)\cap B\left(\frac{Q}{a}\alpha,\left|\frac{Q}{a}\right|r\right)\neq \emptyset
\quad\Leftrightarrow\quad
r<\Vert \alpha\Vert\leq \frac{\max\left\{|a|,\left|{Q}\right|\right\}}{\left|a-{Q}\right|}r.
\end{equation}
Observe that if $|a|\neq \left|{Q}\right|$, then \eqref{lem:UM=} implies that the right hand side of \eqref{eqn:admissRange} is equal to $r$, so that $B(\alpha,r)\cap B\left(\frac{Q}{a}\alpha,\left|\frac{Q}{a}\right|r\right)=\emptyset$. Therefore, the only case in which $B\left(\alpha,r\right)\cap B\left(\frac{Q}{a}\alpha,\left|\frac{Q}{a}\right|r\right)\neq \emptyset$ is when $\left|a-Q\right|\le\frac {|Q|}{\|\alpha\|/r}<\vert Q\vert$.

Let $t\ge 1$ be a positive integer such that $\frac {\|\alpha\|}{r} = q^t>1$. By \eqref{eqn:sumVOlBND}, for $Q\in \mathcal{R}_{\text{monic}}$ with $\deg(Q)=d$, we would like to count
\begin{equation}
\label{eqn:|a-Q|<q^d-tCoPrime}
\#\left\{a\in \mathcal R_{=d}: |a-Q|\le q^{d-t},\;\gcd(a,Q)=1 \right\}.
\end{equation}
First, note that \eqref{eqn:|a-Q|<q^d-tCoPrime} is zero whenever $\deg(Q)<t$, since $|a-Q|\ge 1$ provided that $a\neq Q$.
For $Q\in \Rmo$ with $d=\deg(Q)\ge t$, denote the distinct irreducible factors of $Q$ by $P_1,\ldots, P_{\ell}$, where $\ell\ge 1$. It follows that for $1\le i_1 < \ldots < i_k\le \ell$, 
\[\begin{split}
\#\left\{a\in \mathcal R: |a-Q|\le q^{d-t} \right\}\cap P_{i_1}\cdots P_{i_k}\mathcal R
&=\#\left\{a'\in \mathcal R: \left|a'-\frac{Q}{P_{i_1}\cdots P_{i_k}}\right|\le \frac{q^{d-t}}{P_{i_1}\cdots P_{i_k}}\right\}\\
&=\frac{q^{d-t+1}}{|P_{i_1}\cdots P_{i_k}|},
\end{split}\]
provided that we put $a=P_{i_1}\cdots P_{i_k}a'\in P_{i_1}\cdots P_{i_k}\mathcal R$ and noting that $Q/(P_{i_1}\cdots P_{i_k})$ is again a monic polynomial. Using the inclusion-exclusion principle in a similar manner as in \eqref{eqn:CoprimeGenDeg}, we obtain that
\[\begin{split}
    &\#\left\{a\in \mathcal R: |a-Q|\le q^{d-t}\;\text{and}\; \gcd(a,Q)=1 \right\}\\
    &=\#\left\{a\in \mathcal R: |a-Q|\le q^{d-t}\right\}
   +\sum_{k=1}^\ell (-1)^k \sum_{1\le i_1 < \cdots < i_k \le \ell}
    \#\left\{a\in P_{i_1}\cdots P_{i_k} \mathcal R: |a-Q|\le q^{d-t}\right\}\\
    &=q^{d-t+1}\left(1+\sum_{k=1}^\ell (-1)^k \sum_{1\le i_1< \cdots < i_k\le \ell} \frac 1 {|P_{i_1} \cdots P_{i_k}|}\right)
    =q^{d-t+1} \frac {\varphi(Q)}{|Q|}\\
    &=\frac{\varphi(Q)}{q^{t-1}}.
\end{split}\]

Consequently,
\[\begin{split}
\operatorname{Var}(X_{B})
&=\sum_{d=t}^\infty \frac {q^n}{q^{dn}} r^n
\sum_{\scriptsize \begin{array}{c}
Q\in \Rmo\\
\deg(Q)=d\end{array}} q^{-t+1}\varphi(Q)\\
&=q^{n-t+1}r^n(1-q^{-1})\sum_{d=t}^\infty \frac 1 {q^{(n-2)d}} 
=\frac {q^{2n-1}}{q^{n-2}-1}(1-q^{-1})\left(\frac{r}{\Vert \alpha\Vert}\right)^{n-1}\vol(B).
\end{split}\]
\end{proof}
\section{Effective Oppenheim Conjecture}
\label{sec:Oppenheim}
To prove Theorem \ref{thm:N_QCount}, we use Theorem \ref{thm:Rogers} to obtain bounds on measure of the set of lattices with large discrepancy. 
\subsection{Discrepancy Bounds}
For a lattice $\Lambda$ and a set $B\subseteq \mathcal{K}_{\infty}^n$, define the \emph{discrepancy of $\Lambda$ with respect to $B$} as 
$$D(\Lambda,B)=\left|\#(\Lambda\setminus\{0\})\cap B-q^n\operatorname{Vol}(B)\right|.$$
Note that the second moment of the discrepancy $D(\Lambda, B)$ is equal to the variance of the Siegel transform of the indicator function of $B\subseteq \mathcal{K}_{\infty}^n$, thus Theorem~\ref{thm:Rogers} is reformulated as follows.
\begin{lemma}
\label{lem:DiscInt}
For every measurable set $B$, we have 
     $$\int_{\mathcal{L}_n}D(\Lambda,B)^2dm_{\mathcal{L}_n}(\Lambda)\le C_{n}\operatorname{Vol}(B),$$
where $C_n$ is defined in \eqref{eqn:C_n}.
\end{lemma}
We now use Lemma \ref{lem:DiscInt} to bound the measure of the set of matrices whose image in $\mathcal{L}_n$ has high discrepancy with respect to $B$. Towards this end, for $T>0$, a measurable set $B\subseteq \mathcal{K}_{\infty}^n$, and a compact set $\mathcal{C}\subseteq \mathrm{G}$, define
$$\mathcal{M}_{B,T}^{\mathcal{C}}=\{g\in \mathcal{C}:D(g\mathcal{R}^{n},B)\geq T\}.$$
\begin{lemma}
\label{lem:measLargDiscrep}
Let $\mathcal C$ be a compact set in $\mathrm{G}$.
    For every $T\in \mathbb{R}_{>0}$ and measurable $B\subseteq \mathcal{K}_{\infty}^n$, we have
    \begin{equation}
    \label{eqn:M_B,TBND}
        m_{\mathrm{G}}(\mathcal{M}_{B,T}^{\mathcal{C}})\ll_{\mathcal{C}}\frac{\operatorname{Vol}(B)}{T^2}.
    \end{equation}
\end{lemma}
\begin{proof}
    Equation \eqref{eqn:M_B,TBND} follows directly from Lemma~\ref{lem:DiscInt} (see the proof of \cite[Lemma 2.2]{KY} for details).
\end{proof}

\subsection{Volume Estimate}
\label{subsec:Vol}
For an isotropic quadratic form $Q$ on $\mathcal K_\infty^n$, denote
\begin{equation*}
\Cone(Q)=\left\{\mathbf v\in \mathcal K_\infty^n: Q(\mathbf v)=0 \right\},
\end{equation*}
followed from the light cone of an indefinite quadratic form over the real field.
\begin{theorem}\label{volume estimate}
Let $Q_0$ be a non-degenerate isotropic quadratic form given as
\begin{equation}\label{standard quadratic form}
Q_0(v_1, \ldots, v_n)=2v_1v_n+Q'(v_2, \ldots, v_{n-1}),
\end{equation}
where $Q'$ is some quadratic form of $(n-2)$ variables. Assume that the absolute values of all the coefficients of $Q_0$ are bounded by $q^N$ for some $N\ge 0$.
Let $h\in C_c^{\infty}(\mathcal K_\infty^n)$ be such that $\supp h \subseteq x^{R}\mathcal O^n\setminus x^{r-1}\mathcal O^n$, where $r\le R\in \mathbb Z$, and
\begin{equation}\label{eqn: Sobolev}
h(\mathbf v + x^{-K}\mathbf{u})=h(\mathbf v), \;\forall \mathbf v\in \mathcal K_\infty^n,\forall \mathbf{u}\in \mathcal{O}^n
\end{equation}
for some positive integer $K\in \mathbb N$. We may assume that $K\ge -r$.
Let $I=a+x^b\mathcal O\subseteq \mathcal K_\infty$ and let $B=\max\{\log_q|a|, b\}$. Then, for any $t> \frac{\max\left\{-2r+B,N+B-b,-3r+B+K\right\}} 2$, we have
\begin{equation}\begin{split}
\label{eqn:smoothVolEst}
\int_{\mathcal K_\infty^n} h(x^{-t}\mathbf v) \chi^{}_{I}(Q_0(\mathbf v)) d\mathbf v
&=J_{h,Q_0}q^{b}q^{(n-2)t},\\
\end{split}\end{equation}
where $J_{h,Q_0}>0$ is a constant depending only on $h$ and $Q_0$.

As a consequence, for any quadratic form $Q\in \{Q_0^g: g\in \SL_n(\mathcal K_\infty)\}$ there exists a constant $c_Q>0$, such that for any $I$ as above and for any $T>\frac{1}{2}\max\left\{0,B,N+B-b,B+K\right\}$, where $B,N,K$ are as above, we have
\begin{equation}
\label{eqn:VolEst}
\vol\left\{\mathbf v \in \mathcal K_\infty^n: \|\mathbf v \|\leq q^T\;\text{and}\; Q(\mathbf v)\in I\right\}
=c_Q \vol(I)q^{(n-2)T}+O_{Q,I}\left(1\right).
\end{equation}
\end{theorem}
\begin{remark}
The condition \eqref{eqn: Sobolev} for a smooth function $h$ substitutes the property regarding the Sobolev norm used in the statement of \cite[Lemma 3.1]{KY}. 
\end{remark}

We first prove the theorem in the special case where $R=r=0$.
\begin{theorem}
\label{thm:VolR=0}
    Let $Q_0$ be as in Theorem \ref{volume estimate}. Let $h\in C_c^{\infty}(\mathcal{K}_{\infty}^n)$ be such that $\operatorname{supp}(h)\subseteq \mathcal{O}^n\setminus \mathfrak{m}^n$, and assume that there exists some $K\in \mathbb{N}$, such that \eqref{eqn: Sobolev} holds. Let $I=a+x^{b}\mathcal{O}\subseteq \mathcal{K}_{\infty}$, and let $B=\max\{\log_q\vert a\vert,b\}$. Then, there exists $J_{h,Q_0}$, such that for any $t>\frac{1}{2}\max\{B,N+B-b,B+K\}$, we have
    $$\int_{\mathcal{K}_{\infty}^n}h(x^{-t}\mathbf{v})\chi_I(Q_0(\mathbf{v}))d\mathbf{v}=J_{h,Q_0}q^bq^{(n-2)t}.$$
\end{theorem}
To prove Theorem~\ref{thm:VolR=0}, we will use the following function field analogue of Witt's theorem, whose proof is delayed to \cref{sec:Witt} to continue the flow of this section.
\begin{theorem}\label{thm: Witt}
    Let $\alpha\in \mathcal{K}_{\infty}$ and let $z\in\mathbb{Z}$. The group $\mathrm{K}=\operatorname{SO}(Q)\cap \operatorname{SL}_n(\mathcal{O}_{\infty})$ acts transitively on the set
    $$L(\alpha,q^z):=\left\{\mathbf{v}\in \mathcal{K}_{\infty}^n:Q(\mathbf{v})=\alpha,\Vert \mathbf{v}\Vert=q^z\right\}.$$
\end{theorem}
\begin{proof}[Proof of Theorem \ref{thm:VolR=0}]
Since smooth functions can be approximated with sums of finitely many characteristic functions, we may assume that $h$ is the characteristic function of a Borel set $A\subseteq \mathcal{O}^n\setminus \mathfrak{m}^n$ satisfying
$$A+x^{-K}\mathcal{O}^n=A.$$
By defining $x^{-t}\mathbf v=\mathbf w$, we have
\begin{equation}\label{eqn 1: volume estimate1}
\int_{\mathcal K_\infty^n} h(x^{-t}\mathbf v) \chi^{}_I(Q_0(\mathbf v)) d\mathbf v
=q^{nt}\int_{\mathcal K_\infty^n} h(\mathbf w) \chi^{}_I(x^{2t}Q_0(\mathbf w)) d\mathbf w.
\end{equation}
Set
\[
C_t(A,I)=\left\{\mathbf w\in \mathcal K_\infty^n: \mathbf w \in A\;\text{and}\;  x^{2t}Q_0(\mathbf w)\in I\right\},
\]
so that
\[
\eqref{eqn 1: volume estimate1}= q^{nt} \vol(C_t(A,I)).
\]
\begin{claim}
For every $t>\frac{1}{2}\max\{B+K, B+N\}$, we have
\[
C_t(A,I)=C'_t(A,I),
\]
where
\[
C'_t(A,I)=\left\{\mathbf u_1+ x^{-2t} \mathbf u_2: \begin{array}{c}
\mathbf u_1\in A\cap \Cone(Q_0);\\[0.05in]
\|\mathbf u_2\|\le q^B,\; 2Q_0(\mathbf u_1, \mathbf u_2)\in I
\end{array}\right\},
\]
where by abuse of notation, $Q_0$ denotes the bilinear form corresponding to the quadratic form $Q_0$, that is, $$Q_0(\mathbf v, \mathbf w)=\frac 1 2\left(Q_0(\mathbf v+\mathbf w)-Q_0(\mathbf v)-Q_0(\mathbf w)\right),$$
so that $Q_0(\mathbf{v},\mathbf{v})=Q_0(\mathbf{v})$.
\end{claim}
\begin{proof}
The fact that $C'_t(A,I) \subseteq C_t(A,I)$ follows from the following observations: Let $\mathbf{u}_1+x^{-2t}\mathbf{u}_2\in C_t'(A,I)$.
\begin{itemize}
\item Since $2t>B$ and $\|\mathbf u_2\|\le q^B$, we have 
\[\|x^{-2t} \mathbf u_2\|<1=\|\mathbf u_1\|.\] 
Thus, by \eqref{lem:UM=}, $\Vert \mathbf{u}_1+x^{-2t-z}\mathbf{u}_2\Vert=1$.
\vspace{1pt}
\item Since $2t> B+K$, we have $x^{-2t}\mathbf u_2 \in x^{-K}\mathcal O^n$. Thus, the fact that $\mathbf{u}_1\in A$ and \eqref{lem:UM=} imply that $\mathbf{u}_1+x^{-2t}\mathbf{u}_2\in A$. 
\vspace{1pt}
\item Since $2t\geq N+B-b$, we have 
$$x^{-2t}|Q_0(\mathbf u_2)|\leq q^{-2t+N}\Vert \mathbf{u}_2\Vert^2\le q^{-2t+N+2B}\le q^b.$$ Hence, by \eqref{lem:UM},
\begin{equation}
\begin{split}
\left| x^{2t}Q_0(\mathbf u_1 + x^{-2t}\mathbf u_2)-a\right|
=\left| 2Q_0(\mathbf u_1, \mathbf u_2)+x^{-2t}Q_0(\mathbf u_2)-a\right|\\
\leq \max\left\{\left| 2Q_0(\mathbf{u}_1,\mathbf{u}_2)-a\right|,\left|x^{-2t}Q_0(\mathbf{u}_2)\right|\right\}\leq q^b.
\end{split}
\end{equation}
\end{itemize}
Thus, $x^{2t}Q_0(\mathbf{u}_1+x^{-2t}\mathbf{u}_2)\in I$, so that $\mathbf{u}_1+x^{-2t}\mathbf{u}_2\in C_t(A,I)$. 
\vspace{5pt}

\noindent To prove that $C_t(A,I)\subseteq C_t'(A,I)$, note that by Theorem~\ref{thm: Witt}, the group $\mathrm{K}=\SO(Q_0)\cap \operatorname{SL}_n(\mathcal{O})$ acts transitively on the set $\left\{\mathbf{v}\in \mathcal{K}_{\infty}^n:Q_0(\mathbf{v})=Q_0(\mathbf{w}),\Vert \mathbf{v}\Vert=\Vert \mathbf{w}\Vert\right\}$. Moreover, by \eqref{standard quadratic form},
$$Q_0\left(\mathbf{e}_1+\frac{Q_0(\mathbf{w})}{2}\mathbf{e}_n\right)=2\cdot \frac{Q_0(\mathbf{w})}{2}=Q_0(\mathbf{w}).$$
In addition, note that $$\vert Q_0(\mathbf{w})\vert\leq \max\left\{q^{-2t}\vert a\vert,\left|Q_0(\mathbf{w})-\frac{a}{x^{2t}}\right|\right\}\leq \max\{q^{B-2t},q^{b-2t}\}<1.$$ Thus, by \eqref{lem:UM=}, $\Vert\mathbf{e}_1+Q_0(\mathbf{w})\mathbf{e}_n\Vert=\Vert \mathbf{e}_1\Vert$, and hence, $\|\mathbf u_1+x^{-2t}\mathbf u_2\|=\|\mathbf u_1\|=1$.
Thus, for every $\mathbf w\in C_t(A,I)$, there exists $k\in \mathrm{K}$ such that
\[
k^{-1}\mathbf w=\mathbf{e}_1+\frac{Q_0(\mathbf{w})}{2}\mathbf{e}_n=\mathbf e_1+ x^{-2t}\:\frac {x^{2t}Q_0(\mathbf w)} 2\mathbf e_n.
\]
Set $\mathbf u_1=k\mathbf e_1$ and $\mathbf u_2=\frac {x^{2t}Q_0(\mathbf w)} 2 k\mathbf e_n$.

\begin{itemize} 
\item Since $Q_0(\mathbf w)\in x^{-2t}a+x^{-2t+b}\mathcal O\subseteq x^{-2t+B}\mathcal{O}$,
\[
\|\mathbf u_2\|=\left| \frac {x^{2t}Q_0(\mathbf w)}{2}\right|\cdot \Vert k\mathbf{e}_n\Vert\le q^{2t}q^{-2t+B}\cdot 1=q^B.
\]
\vspace{1pt}
\item Thus, for $2t>B+K$, since $\mathbf u_1=\mathbf w- x^{-2t}\mathbf u_2=k\mathbf{e}_1$ and $\mathbf w\in A$,
\[
\mathbf u_1 \in A \cap \Cone(Q_0).
\]
Clearly, $2Q_0(\mathbf u_1, \mathbf u_2)=Q_0(\mathbf w)\in I$.
\end{itemize}
Hence, $\mathbf{w}=\mathbf{u}_1+x^{-2t}\mathbf{u}_2\in C_t(A,I)$.
\vspace{5pt}

\noindent We now compute the volume of $C'_t(A,I)$. Note that the volume on $\mathcal K_\infty^n$ is the limit of the normalized counting measures of the image set of the projection $\mathcal K_\infty^n \rightarrow \mathcal K_\infty^n/ x^{-\ell}O^n$ for $\ell\in \N$, i.e.,

\begin{equation}
\begin{split}\label{eqn:Vol(C_t'(A,I)))}
&\vol\left(C'_t(A,I)\right)\\
&=\lim_{\ell\rightarrow \infty} \frac 1 {q^{n\ell}} \:
 \#\left\{[\mathbf u_1 + x^{-2t} \mathbf u_2]_\ell: 
 \begin{array}{c}
 \mathbf u_1\in A\cap \Cone(Q_0),\\
 {[\mathbf{u}_2]_{\ell}\in [x^B\mathcal{O}^n]_{\ell}},\; [2Q_0(\mathbf u_1, \mathbf u_2)]_{\ell}\in [I]_{\ell}\end{array}
 \right\}\\[0.1in]
 &=\lim_{\ell\rightarrow \infty} \frac 1 {q^{n\ell}}
\sum_{\scriptsize\begin{array}{r}
[\mathbf{u}_1]_{\ell}\in [\mathcal{O}^n\setminus x^{-1}\mathcal{O}^n]_{\ell}\\
\cap [A\cap \operatorname{Cone}(Q_0)]_{\ell}\end{array}}
\frac{\# 
 \left\{[x^{-2t}\mathbf u_2]_\ell\in [x^{B-2t}\mathcal{O}^n]_{\ell}: 2Q_0(\mathbf u_1, \mathbf u_2)\in I \right\}
}{\# 
    \left\{\begin{array}{l}[(\mathbf{v}_1,\mathbf{v}_2)]_{\ell}\in [A\cap \operatorname{Cone}(Q_0)]_{\ell}\times [x^B\mathcal{O}^n]_{\ell}: \\[0.05in]
\hspace{0.2in}\mathbf v_1+x^{-2t}\mathbf v_2\in C'_t(A,I)\;\text{s.t.}\\[0.05in]
\hspace{1.2in}\mathbf v_1+x^{-2t}\mathbf v_2=\mathbf u_1+x^{-2t}\mathbf u_2\end{array}\right\}
},
\end{split}\end{equation}
where $[\mathbf u]_\ell$ is the equivalence class of $\mathbf u\in \mathcal K_\infty^n $ modulo $x^{-\ell}$ and $[A]_{\ell}=\{[\mathbf v]_\ell: \mathbf u\in A\}$. We now compute the numerator and denominator of the last line of \eqref{eqn:Vol(C_t'(A,I)))} through the following two facts.

\vspace{0.1in}
\noindent i) For any $\mathbf u_1\in A\cap \Cone(Q_0)$ and $\ell>2t-b$, we have 
\[
\#\{[x^{-2t}\mathbf u_2]_\ell: 2Q_0( \mathbf u_1, \mathbf u_2)\in I\}=q^{(-2t+\ell)n+B(n-1)}\cdot \vol(I).
\] 
\begin{proof}
By Theorem \ref{thm: Witt}, $\mathrm{K}$ acts transitively on $\left\{\mathbf{v}:\Vert \mathbf{v}\Vert=1\right\}\cap \Cone(Q_0)$. Thus, there exists $k\in \mathrm{K}$, such that $\mathbf{u}_1=k\mathbf{e}_1$. Hence, by replacing $A$ with the set $kA$,  we can assume that $\mathbf u_1=\mathbf e_1$. 

Denote $\mathbf u_2=(u_1, \ldots, u_n)$. By \eqref{standard quadratic form},
\[
2Q_0(\mathbf e_1, \mathbf u_2)\in I
\quad\Leftrightarrow\quad
2u_n\in I,
\]
that is, $u_n\in \frac a 2  + x^{b}\mathcal O$. Thus 
\[\begin{split}
&\#\{[x^{-2t}\mathbf u_2]_\ell: [2Q_0( \mathbf u_1, \mathbf u_2)]_{\ell}\in [I]_{\ell}\}
=\#\left[x^{-2t}\left(x^B\mathcal O^{n-1}\times \left(\frac a 2 +x^{b}\mathcal O\right)\right)\right]_{\ell}\\
&\hspace{0.4in}=q^{(-2t+\ell)n}\cdot q^{B(n-1)} \cdot q^{b}=q^{(-2t+\ell)n}\cdot q^{B(n-1)} \cdot \:\vol(I).
\end{split}\]
\end{proof}

\vspace{0.1in}
\noindent ii) For any $\mathbf u_1,\mathbf v_1$ in $A\cap \Cone(Q_0)$ and $\mathbf{u}_2,\mathbf{v}_2\in x^B\mathcal{O}^n$, such that $\mathbf{u}_1+x^{-2t}\mathbf{u}_2\in C_t'(A,I)$, we have
\begin{equation}
\label{eqn:u_1=v_1 modx^-2t+B}
\mathbf u_1 + x^{-2t}\mathbf u_2=\mathbf v_1 + x^{-2t}\mathbf v_2\;
\quad\Leftrightarrow\quad
\mathbf u_1\equiv \mathbf v_1 \mod x^{-2t+B}.
\end{equation}
By Theorem \ref{thm: Witt}, by replacing $A$ with $kA$ for some $k\in \mathrm{K}$, we can assume that $\mathbf u_1=\mathbf e_1$ and compute the cardinality of the equivalence class of $\mathbf e_1$ modulo $x^{-2t+B}$ which intersects $\Cone(Q_0)$. Thus, by \eqref{eqn:u_1=v_1 modx^-2t+B},
\begin{equation}\label{eqn 2: volume estimate1}\begin{split}
&\# \left\{\left([\mathbf v_1]_{\ell},[x^{-2t}\mathbf v_2]_\ell\right): \begin{array}{c}
[\mathbf v_1+x^{-2t}\mathbf v_2]_{\ell}\in [C'_t(A,I)]_{\ell}\;\text{s.t.}\\[0.05in]
[\mathbf v_1+x^{-2t}\mathbf v_2]_{\ell}=[\mathbf u_1+x^{-2t}\mathbf u_2]_{\ell}\\
\text{ and } [\mathbf{v}_1]_{\ell}\in [\operatorname{Cone}(Q_0)]_{\ell}\end{array}\right\}\\[0.05in]
&\hspace{1in}=\#\left\{[ \mathbf e_1+x^{-2t+B}\mathbf z]_\ell: \|\mathbf z\|\le 1 \right\}\cap \Cone(Q_0).
\end{split}\end{equation}

By expanding $Q_0( \mathbf e_1+x^{-2t+B}\mathbf z)=0$, the vector $\mathbf z=(z_1, \ldots, z_n)$ must satisfy the following equation 
\[
x^{-2t+B}z_n=\frac {-Q'(x^{-2t+B}z_2, \ldots,x^{-2t+B}z_{n-1})}{2(1+x^{-2t+B}z_1)}.
\]
Hence 
\[
\eqref{eqn 2: volume estimate1}=q^{(-2t+B+\ell)(n-1)}.
\]

As a consequence, 
\[\begin{split}
\eqref{eqn 1: volume estimate1}
&=q^{nt}\vol\left(C'_t(A,I)\right)\\
&=\lim_{\ell\rightarrow \infty}\frac {q^{nt}}{q^{n\ell}}\cdot \frac{\#[A\cap \Cone(Q_0)]_\ell\times q^{(-2t+\ell)n+B(n-1)}\cdot \vol(I)}{q^{(-2t+B+\ell)(n-1)}}\\
&=q^{(n-2)t}\: \vol(I)\lim_{\ell\rightarrow \infty} \frac {\#[A\cap \Cone(Q_0)]_\ell}{q^{(n-1)\ell}}
\end{split}\]
\end{proof}
The asymptotic formula \eqref{eqn:smoothVolEst} follows if we define
\[
J_{h,Q_0}=\lim_{\ell\rightarrow \infty} \frac {\#[A\cap \Cone(Q_0)]_\ell}{q^{(n-1)\ell}},
\]
where $J_{h,Q_0}$ is well-defined for general $h\in C_c(\mathcal O^n_\infty\setminus x^{-1}\mathcal O^n)$.
\end{proof}
\begin{proof}[Proof of Theorem \ref{volume estimate}]
We may assume that $h$ is the characteristic function of a Borel set $A\subseteq \mathcal K_\infty^n$ satisfying that
\[
A+ x^{-K}\mathcal O^n=A
\]
and that $\supp h\subseteq x^R\mathcal O^n\setminus x^{r-1}\mathcal O^n$.
For each $r\le z \le R$, let 
\[
h_z(\mathbf v)=h(\mathbf v)\cdot \chi^{}_{x^z\mathcal O^n\setminus x^{z-1}\mathcal O^n}(\mathbf v).
\]
Define $h'_z(\mathbf v):=h_z(x^{z}\mathbf v)$ for all $\mathbf v\in \mathcal K_\infty^n$ so that $\supp h'_z\subseteq \mathcal O^n\setminus \mathfrak{m}^n$ for all $r\le z\le R$.
It follows that
\[\begin{split}
\int_{\mathcal K_\infty^n} h(x^{-t}\mathbf v) \chi_I(Q_0(\mathbf v))d\mathbf v
&=\int_{\mathcal K_\infty^n} \sum_{z=r}^R h_z(x^{-t}\mathbf v) \chi_I (Q_0(\mathbf v)) d\mathbf v\\
&=\sum_{z=r}^R q^{zn}\int_{\mathcal K_\infty^n} h'_z(x^{-t}\mathbf v)\chi^{}_{I_z}(Q_0(\mathbf v)) d\mathbf v,\\
\end{split}\]
where $I_z=x^{-2z}a+x^{b-2z}\mathcal O\subset \mathcal K_\infty$. Moreover, for each $z$,
\[\begin{gathered}
h'_z(\mathbf v+x^{-K-z} \mathcal O^n)
=h_z(x^{z}\mathbf v + x^{-K}\mathcal O^n)
=h_z(x^{z}\mathbf v)
=h'_z(\mathbf v),\quad\forall \mathbf v\in \mathcal K^n_\infty.\\
\end{gathered}\]
Applying \eqref{eqn:smoothVolEst} to each $h'_z$, if $2t>\max\{-2z+B, N+B-b, -z+B+K\}$, we have
\[
\int_{\mathcal K_\infty^n} h'_z(x^{-t}\mathbf v) \chi^{}_{I_z}(Q_0(\mathbf v)) d\mathbf v
=J_{h'_z,Q_0}q^{b-2z}q^{(n-2)t}.
\]
Thus, if $2t>\max\{-2r+B, N+B-b,-3r+B+K\}$, one has
\[
\int_{\mathcal K_\infty^n} h(x^{-t}\mathbf v) \chi_I(Q_0(\mathbf v))d\mathbf v
=\sum_{z=r}^R q^{zn} J_{h'_z,Q_0}q^{b-2z}q^{(n-2)t}.
\]
In conclusion, if $2t>\max\{-2r+B, N+B-b,-3r+B+K\}$, the asymptotic formula \eqref{eqn:smoothVolEst} holds
if we take
\begin{equation}
\label{eqn:J_Q}
J_{h,Q_0}=\sum_{z=r}^R q^{z(n-2)}J_{h'_z,Q_0}.
\end{equation}
\vspace{0.1in}
Let $Q=Q_0\circ g$ for some $g\in \SL_n(\mathcal K_\infty)$. Let $h=\chi^{}_{g(\mathcal O^n\setminus x^{-1}\mathcal O^n)}$, so that for $$t_Q=\frac{1}{2}\max\{-2r+B,N+B-b,-3r+B+K\}\in \mathbb N$$ and $J_{Q}>0$ defined as in \eqref{eqn:J_Q}, for every $t\ge t_Q$, we have
\[
\vol\left\{\mathbf v\in \mathcal K_\infty^n: \|\mathbf v\|=q^t,\; Q(\mathbf v) \in I \right\}
=J_{Q}\vol(I) q^{(n-2)t}.
\]
Let 
\[
\mathrm{V}_{t_Q}:= \vol\left\{\mathbf v\in \mathcal K_\infty^n: \|\mathbf v\|\le q^{t_Q},\; Q(\mathbf v)\in I\right\}.
\]
For $T> t_Q$, it follows that
\[\begin{split}
&\vol\left\{\mathbf v\in \mathcal K_\infty^n: \|\mathbf v\|\le q^{t} ,\; Q(\mathbf v)\in I\right\}\\
&=\vol\left\{\mathbf v\in \mathcal K_\infty^n: q^{t_Q}<\|\mathbf v\|\le q^{t} ,\; Q(\mathbf v)\in I\right\}+\mathrm{V}_{t_Q}\\
&=\sum_{t=t_Q+1}^T J_Q\vol(I)q^{(n-2)t}+\mathrm{V}_{t_Q}\\
&=\frac {q^{n-2}J_Q}{q^{n-2}-1}\cdot \vol(I)q^{(n-2)T}+\left(\mathrm{V}_{t_Q}-\frac {q^{n-2}J_Q}{q^{n-2}-1}\cdot \vol(I)q^{(n-2)t_Q}\right)
\end{split}\]
Thus, we obtain the volume formula by taking $c_Q=\frac{q^{n-2}J_Q}{q^{n-2}-1}$. 
\end{proof}
\subsection{Proof of Theorem \ref{thm:N_QCount}}
Recall that if $Q$ is a non-degenerate and isotropic quadratic form, there exist $g\in \mathrm{G}$ and a quadratic form $Q_0$ of the form \eqref{standard quadratic form} such that $Q=Q_0\circ g$. Hence, for such $Q_0$, the set of $Q=Q_0\circ g$, where $g\in \mathrm{G}$, can be identified with $\mathcal{L}_n$ via $Q_0\circ g\mapsto g\mathcal{R}^n$. For such $Q$ and a measurable set $I\subseteq \mathcal{K}_{\infty}$, define
\begin{align*}
N_Q(I,t)&=\#\left\{\mathbf{v}\in \mathcal{R}^n:Q(\mathbf{v})\in I,\Vert \mathbf{v}\Vert\leq q^t\right\}
\\&=\#\{\mathbf{u}\in g\mathcal{R}^n:Q_0(\mathbf{u})\in I,\mathbf{u}\in gB(0,q^t)\}\\
&=\#g\mathcal{R}^n\cap Q_0^{-1}(I)\cap gB(0,q^t)
=\#g\mathcal R^n \cap A_{g,I,t},
\end{align*}
where $A_{g,I,t}:=Q_0^{-1}(I)\cap gB(0,q^t)$.
\begin{remark}
\label{rem:Covering}
Observe that $\SL_n(\mathcal O)$ is an open subset in $\SL_n(\mathcal K_\infty)$ which preserves the balls centered at the origin.
For any given compact set $\mathcal C\subseteq G=\SL_n(\mathcal K_\infty)$, there exist $\eta>0$ and $h_1,\dots,h_{\eta}\in \mathcal{C}$ such that $\mathcal{C}\subseteq \bigcup_{i=1}^{\eta}h_i\operatorname{SL}_n(\mathcal{O})$. Hence, for any $g\in h_i \SL_n(\mathcal O)$, we have
\[
N_Q(I,t)=\# g\mathcal R^n \cap A_{g,I,t}
=\# g\mathcal R^n \cap A_{h_i,I,t}.
\]
\end{remark}
\begin{theorem}
\label{thm:Disca.e.Small}
    Let $n\ge 3$, let $\frac 1 2 < \delta < 1$, and let $I\subseteq \mathcal{K}_{\infty}$ be a measurable set. Then for $m_\mathrm{G}$-almost every $g\in \mathrm{G}$, there exists $t_{g,I}>0$ so that for every $t\geq t_{g,I}$, we have
    $$D\left(g\mathcal{R}^n,A_{g,I,t}\right)<\operatorname{Vol}(A_{g,I_t,t})^{\delta}.$$
\end{theorem}
\begin{proof}
    For a fixed compact set $\mathcal{C}\subseteq \mathrm{G}$ and $t\in \mathbb{N}$, define $\mathcal{B}_t\subseteq \mathcal{C}$ by
    $$\mathcal{B}_t=\left\{g\in \mathcal{C}:D\left(g\mathcal{R}^n,A_{g,I,t}\right)\geq \operatorname{Vol}(A_{g,I,t})^{\delta}\right\}.$$
    We prove that $m_{\mathrm{G}}\left(\limsup_{t\rightarrow \infty}\mathcal{B}_t\right)=0$, so that for $m_\mathrm{G}$-almost every $g\in \mathcal{C}$, we have $D\left(g\mathcal{R}^n,A_{g,I_t,t}\right)<\operatorname{Vol}(A_{g,I_t,t})^{\delta}$ for every sufficiently large $t$. Since this holds for any compact set $\mathcal{C}$, this concludes the proof. Note that
    $$m_{\mathrm{G}}\left(\limsup_{t\rightarrow \infty}\mathcal{B}_t\right)=\lim_{m\rightarrow \infty}m_{\mathrm{G}}\left(\bigcup_{t\geq m}\mathcal{B}_t\right)\leq \lim_{m\rightarrow\infty}\sum_{t=m}^{\infty}m_{\mathrm{G}}\left(\mathcal{B}_t\right).$$
    We now prove that $\sum_{t>0}m_{\mathrm{G}}(\mathcal{B}_t)<\infty$, so that the Borel Cantelli lemma implies that $m_{\mathrm{G}}(\limsup \mathcal{B}_t)=0$. By Remark \ref{rem:Covering}, there exist $h_1,\dots,h_{\eta}\in \mathcal{C}$ such that $\mathcal{C}\subseteq \bigcup_{i=1}^{\eta}h_i\operatorname{SL}_n(\mathcal{O})$. Thus, for every $g\in \mathcal C$, there exists $ i\in\left\{1,\dots, \eta\right\}$, such that $g\in h_i\operatorname{SL}_n(\mathcal{O})$. For $i=1,\dots, \eta$, define
    \begin{equation}
        \mathcal{B}_{t,i}:=\left\{g\in
        h_i\operatorname{SL}_n(\mathcal{O}):D\left(g\mathcal{R}^n,A_{h_i,I,t}\right)\geq \operatorname{Vol}(A_{h_i,I,t})^{\delta}\right\}.
    \end{equation}
    By Lemma \ref{lem:DiscInt} and Theorem \ref{volume estimate}, there exists $t_{h_i,I}$, such that for every $t>t_{h_i,I}$, we have
    \begin{equation}
    \begin{split}
    \label{eqn:m(B_t,i)Bnd}
        m_{\mathrm{G}}(\mathcal{B}_{t,i})\ll_{h_i}\operatorname{Vol}(A_{h_i,I_t,t})^{1-2\delta}
        \ll_{h_i,I} q^{t(n-2)(1-2\delta)}\\
    \end{split}
    \end{equation}
    for every large enough $t$.
    Since $\frac{1}{2}<\delta<1$,
    we have $\sum_{t=0}^{\infty}m_{\mathrm{G}}(\mathcal{B}_{t,i})<\infty$. 
    Since $\mathcal{C}\subseteq \bigcup_{i=1}^{\eta}h_i\operatorname{SL}_n(\mathcal{O})$, it follows that $\sum_{t=0}^{\infty}m_{\mathrm{G}}(\mathcal{B}_t)<\infty$, so that  $m_{\mathrm{G}}(\limsup_{t\rightarrow \infty}\mathcal{B}_t)=0$. 
\end{proof}
We now use Theorem \ref{thm:Disca.e.Small} to prove Theorem \ref{thm:N_QCount}. 
\begin{proof}[Proof of Theorem \ref{thm:N_QCount}]
    Let $\frac{1}{2}<\delta<1$. By Theorem \ref{volume estimate}, Remark \ref{rem:Covering}, and Theorem \ref{thm:Disca.e.Small}, for $m_\mathrm{G}$ almost every $g$, for $Q=Q_0\circ g$, and for every sufficiently large $t$, we have
    \begin{equation}
    \begin{split}
    \label{eqn:N_Q(I_t,t)Cnt}
        &\left|N_Q(I,t)-c_Qq^{(n-2)t+n}m_{\mathcal{K}_{\infty}}(I)\right|\\
        &\hspace{0.4in}\leq D(g\mathcal{R}^n,A_{g,I,t})+q^n\left|\operatorname{Vol}(A_{g,I,t})-c_Qq^{t(n-2)}m_{\mathcal{K}_{\infty}}(I)\right|\\
        &\hspace{0.4in}<\operatorname{Vol}(A_{g,I,t})^{\delta}+O_{Q,I}(1)\\
        &\hspace{0.4in}\ll \left(c_Qm_{\mathcal{K}_{\infty}}(I)q^{t(n-2)}+O_{Q,I}(1)\right)^{\delta}+O_{Q,I}(1)\\
        &\hspace{0.4in}=O_{Q,I}\left(q^{\delta t(n-2)}m_{\mathcal{K}_{\infty}}(I)^{\delta}\right).
    \end{split}
    \end{equation}
\end{proof}
\section{Witt's Theorem in Function Fields}
\label{sec:Witt}
In this section, we prove Theorem \ref{thm: Witt}. The proof is almost identical to the proof of \cite[Proposition A1]{HLM} and is included for completeness. 
\begin{proof}[Proof of Theorem \ref{thm: Witt}]
    We need to prove that for any $\mathbf{v}_1,\mathbf{v}_2\in \mathcal{K}_{\infty}^n$ with $\Vert \mathbf{v}_i\Vert=q^z$ and $Q(\mathbf{v}_i)=\alpha$ for $i=1,2$, there exists $k\in \mathrm{K}$ such that $k\mathbf{v}_1=\mathbf{v}_2$. It suffices to assume that $z=0$ and $\mathbf{v}_1=\mathbf{e}_1$ and $\mathbf{v}_2=\mathbf{v}\in \mathcal{O}_{\infty}^n\setminus \mathfrak{m}^n$, since all other cases follow directly from this case. We may assume that $Q(y_1,\dots,y_n)$ has one of the following forms:
    \begin{equation}
        \begin{split}
            u_1y_1^2+\dots+u_iy_i^2+x^{-1}(u_{i+1}y_{i+1}^2+\dots+u_ny_n^2)\text{  or }\\
            y_1y_2+u_3y_3^2+\dots+u_iy_i^2+x^{-1}(u_{i+1}y_{i+1}^2+\dots+u_ny_n^2),
        \end{split}
    \end{equation}
    where $u_i\in \mathcal{O}_{\infty}\setminus \mathfrak{m}$.
    Then for any $\mathbf{w}\in \mathcal{K}_{\infty}^n$, we have $Q(\mathbf{w})=\mathbf{w}^tY\mathbf{w}$, where  $Y=\begin{pmatrix}
        Y'&\\
        &x^{-1}Y''
    \end{pmatrix}$, where $Y'\in \operatorname{GL}_i(\mathcal{K}_{\infty}),Y''\in \operatorname{GL}_{n-i}(\mathcal{K}_{\infty})$ are nondegenerate modulo $x^{-1}$. The goal is to construct a matrix $k\in \mathrm{K}$ satisfying
    \begin{enumerate}
        \item \label{cond:Transitive} $k\mathbf{e}_1=\mathbf{v}$ and
        \item \label{cond:StabQ} $k^tYk=Y$,
    \end{enumerate}
    where the second condition ensures that $k\in \mathrm{K}$. Since $\mathcal{O}_{\infty}$ is the inverse limit of $\mathcal{R}_x/x^{-j}\mathcal{R}_x$, where $\mathcal{R}_x:=\mathbb{F}_q[x^{-1}]$, we construct a sequence $k^{(j)}\in \operatorname{GL}_n(\mathcal{R}_x/x^{-(j+1)}\mathcal{R}_x)$ satisfying
    \begin{enumerate}
        \item $k^{(j)}\mathbf{e}_{1}=\mathbf{v}\mod x^{-(j+1)}$, 
        \item ${k^{(j)}}^t Y k^{(j)}=B\mod x^{-(j+1)}$, and
        \item $k^{(\ell)}=k^{(j+1)}\mod x^{-(j+1)}$ for every $\ell \leq j$.
    \end{enumerate}
    Then $\lim_{\leftarrow} k_{\ell}$ will be an element satisfying \eqref{cond:Transitive} and \eqref{cond:StabQ}. In this case, we write $k^{(j)}=\sum_{i=0}^jk_ix^{-i}$, where $k_i\in \mathbb{F}_q$, and $\mathbf{v}=\sum_{j=0}^{\infty}x^{-j}\mathbf{v}_j=(v_1,\dots,v_n)$, where $\mathbf{v}_j$ is the projection of $\mathbf{v}$ on $\mathcal{R}_x/x^{-j}\mathcal{R}_x$.
    \paragraph{\textbf{Step 1: $j=0$.}} Let $k^{(0)}=k_0=\begin{pmatrix}
        A_0&B_0\\
        C_0&D_0
    \end{pmatrix}$. Then, $x^{-1}Y''\equiv 0\mod x^{-1}$, so that the matrix $k_0$ should satisfy the following
    \begin{equation}
    \begin{split}
    \label{eqn:WittStep0}
    \begin{pmatrix}
        Y'&\\
        &0
        \end{pmatrix}\equiv \begin{pmatrix}
        A_0^t&C_0^t\\
        B_0^t&D_0^t
    \end{pmatrix}\begin{pmatrix}
        Y'&\\
        &0
    \end{pmatrix}\begin{pmatrix}
        A_0&B_0\\
        C_0&D_0
    \end{pmatrix}\mod x^{-1}\\
    =\begin{pmatrix}
        A_0^tY'A_0&A_0^tY'B_0\\
        B_0^tY'A_0&B_0^tY'B_0
    \end{pmatrix}\mod x^{-1}.
    \end{split}
    \end{equation}
    Let $\operatorname{pr}_i:(\mathcal{R}_x/x^{-1}\mathcal{R}_x)^n\rightarrow (\mathcal{R}_x/x^{-1}\mathcal{R}_x)^i$ be the projection onto the first $i$ coordinates, that is $\operatorname{pr}_i(\alpha_1,\dots,\alpha_n)=(\alpha_1,\dots,\alpha_i)$. By assumption, $Q(\operatorname{pr}_i(\mathbf{e}_1))=Q(\operatorname{pr}_i(\mathbf{v}_0))\mod x^{-1}$. Note that $\mathcal{R}_x/x^{-1}\mathcal{R}_x\cong \mathbb{F}_q$. Therefore, by applying Witt's theorem for finite fields \cite[page 121]{Art} to $\left((\mathcal{R}_x/x^{-1}\mathcal{R}_x)^i,Y'\right)$, we obtain an isometry $A_0$, whose first column is $\operatorname{pr}_i(\mathbf{v}_0)$, and satisfies $A_0^tY'A_0=Y'\mod x^{-1}$. Since $A_0^tY'$ is invertible, by \eqref{eqn:WittStep0}, $B_0=0$. Similarly to \cite[Proposition A.1]{HLM}, at this step, we cannot determine $C_0$ and $D_0$, as they do not appear in the right hand side of \eqref{eqn:WittStep0}. 
    \paragraph{\textbf{Step 2: $j=1$.}} Let $k^{(1)}=\begin{pmatrix}
        A_0+xA_1&B_0+xB_1\\
        C_0+xC_1&D_0+xD_1
    \end{pmatrix} \in \mathcal{R}_x/x^{-2}\mathcal{R}_x$ be a matrix whose first column is $\mathbf{v}_0+x^{-1}\mathbf{v}_1$ and satisfies $(k^{(1)})^tYk^{(1)}=Y$. Hence, 
    \begin{equation}
    \begin{split}
        A_0^tY'A_0+x(A_1^tY'A_0+A_0^tY'A_1+C_0^tY''C_0)\equiv Y'\mod x^{-2},\\
        A_0^tY'B_0^t+x\left[A_1^tY'B_0+A_0^tY'B_1+C_0^tY''D_0\right]\equiv 0\mod x^{-2},\\
        B_0^tY'B_0+x\left[B_1^tY'B_0+B_0^tY'B_1+D_0^tY''D_0\right]\equiv xY''\mod x^{-2}.
    \end{split}
    \end{equation}
    Since $B_0=0$ and $A_0^tY'A_0=Y'+xZ_1^{11}\mod x^2$ for some symmetric matrix $Z_1^{11}$, then, we have
    \begin{equation}
    \begin{split}
    \label{eqn:x^{-1}CONDS}
        Z_1^{11}+A_1^tY'A_0+A_0^tY'A_1+C_0^tY''C_0\equiv 0 \mod x^{-1},\\
        A_0^tY'B_1+C_0^tY''D_0\equiv 0 \mod x^{-1},\\
        D_0^tY''D_0\equiv Y''\mod x^{-1}.
    \end{split}
    \end{equation}
    Let $C_0\in M_{n-i\times i}(\mathcal{R}_x/x^{-1}\mathcal{R}_x)$ be some matrix whose first column is $(v_{i+1},\dots,v_n)$. By Witt's theorem for finite fields \cite[page 121]{Art}, there exists a matrix $D_0\in M_{n-i\times n-i}(\mathcal{R}_x/x^{-1}\mathcal{R}_x)$ satisfying $D_0^tY''D_0=Y''\mod x^{-1}$. Hence, it suffices to prove that there exists $A_1\in M_{i\times i}(\mathcal{R}_x/x^{-1}\mathcal{R}_x)$ such that $Z_1^{11}+A_1^tY'A_0+A_0^tY'A_1+C_0^tY''C_0\equiv 0\mod x^{-1}$. 
    \paragraph{\textbf{Step 3}} We now prove the following claim to show that $$Z_1^{11}+A_1^tY'A_0+A_0^tY'A_1+C_0^tY''C_0\equiv 0 \mod x^{-1}$$ has a solution $A_1$.
    \begin{claim}
    \label{claim:SymmSol}
    For any invertible symmetric matrix $A\in M_{n\times n}(\mathcal{R}_x/x^{-1}\mathcal{R}_x)$ and any symmetric matrix $C\in M_{n\times n}(\mathcal{R}_x/x^{-1}\mathcal{R}_x)$, there exists a solution $X\in M_{n\times n}(\mathcal{R}_x/x^{-1}\mathcal{R}_x)$ to the equation $X^tA+AX=C$. 
    \end{claim}
    \begin{proof}By viewing $n\times n$ matrices as $n^2$ dimensional vectors, the equation $X^tA+AX=C$ can be rewritten as 
    \begin{equation}
    \label{eqn:MatrixEq}
        \begin{pmatrix}
            A_{11}&A_{12}&\dots&A_{1n}\\
            A_{21}&A_{22}&\dots&A_{2n}\\
            \vdots&\vdots&\ddots&\vdots\\
            A_{n1}&A_{n2}&\dots&A_{nn}
        \end{pmatrix}\begin{pmatrix}
            X^{(1)}\\
            X^{(2)}\\
            \vdots\\
            X^{(n)}
        \end{pmatrix}=\begin{pmatrix}
            C^{(1)}\\
            C^{(2)}\\
            \vdots \\
            C^{(n)}
        \end{pmatrix},
    \end{equation}
    where $X^{(j)}=(x_{1j},\dots, x_{nj})^T$, $C^{(j)}=(c_{1j},\dots,c_{nj})^T$, and
    \begin{equation}
        A_{ij}=\begin{cases}
            \begin{pmatrix}
                a_{11}&a_{21}&\dots&\dots&a_{n1}\\
                \vdots&\vdots&\dots&\dots&\vdots\\
                2a_{1i}&2a_{2i}&\dots&\dots&2a_{ni}\\
                \vdots&\vdots&\dots&\dots&\vdots\\
                a_{1n}&a_{2n}&\dots&\dots&a_{nn}
            \end{pmatrix},& \text{when } i=j,\\
            \text{All entries are zero except the }j\text{-th row, which is } (a_{1i},\dots,a_{ni}),&i\neq j.
        \end{cases}
    \end{equation}
    Since $X^tA+AX$ and $C$ are both symmetric then, by removing identical rows from \eqref{eqn:MatrixEq}, one gets a system of linear equations from $(\mathcal{R}_x/x^{-1}\mathcal{R}_x)^{n^2}$ to $(\mathcal{R}_x/x^{-1}\mathcal{R}_x)^{\frac{n(n+1)}{2}}$. Since $A$ is invertible, the rank of this system is $\frac{n(n+1)}{2}$, and therefore there exists a solution to \eqref{eqn:MatrixEq}. 
\end{proof}
    Hence, by applying Claim \ref{claim:SymmSol}, with $C=-Z_1^{11}-C_0^tY''C_0$, the equation $Z_1^{11}+A_1^tY'A_0+A_0^tY'A_1+C_0^tY''C_0\equiv 0\mod x^{-1}$ has a solution $A_1$, where $Z_{1}^{11},Y',Y'',A_0,C_0$ and the first column of $A_1$ are all known. 
    \paragraph{\textbf{Step 4: General Case}} Assume that there exists some solution $k=\sum_{j=0}^{\infty}k_jx^{-j}$ which satisfies \eqref{cond:Transitive} and \eqref{cond:StabQ}. Since $k^tYk=Y$, 
    \begin{equation}
    \begin{split}
        \begin{pmatrix}
            Y'&\\
            &x^{-1}Y''
        \end{pmatrix}=\sum_{j=0}^{\infty}x^{-j}\left(\sum_{i=0}^j\begin{pmatrix}
            A_i^t&C_i^t\\
            B_i^t&D_i^t
        \end{pmatrix}\begin{pmatrix}
            Y'&\\
            &x^{-1}Y''
        \end{pmatrix}\begin{pmatrix}
            A_i&B_i\\
            C_i&D_i
        \end{pmatrix}\right)\\
        =\sum_{j=0}^{\infty}x^{-j}\left(\sum_{i=0}^j\begin{pmatrix}
            A_i^tY'A_{j-i}&A_i^tY'B_{j-i}\\
            B_i^tY'A_{j-i}&B_i^tY'B_{j-i}
        \end{pmatrix}\right)\\
        +\sum_{j=0}^{\infty}x^{-(j+1)}\left(\sum_{i=0}^j\begin{pmatrix}
            C_i^tY''C_{j-i}&C_i^tY''D_{j-i}\\
            D_i^tY''C_{j-i}&D_i^tY''D_{j-i}
        \end{pmatrix}\right).
    \end{split}    
    \end{equation}
    We now find $A_j,B_j,C_j,D_j$ inductively using step 3. Let $C_{j-1}$ be a matrix whose first column equal to $(v_{j-1}^{(i)},\dots,v_{j-1}^{(n)})$, where $\mathbf{v}=\sum_{\ell=0}^{\infty}x^{-\ell}\mathbf{v}_\ell$ and $\mathbf{v}_\ell=(v_\ell^{(1)},\dots,v_{\ell}^{(n)})^t\in \mathbb{F}_q^n$. Then by step 3 and the fact that $\mathbf{v}^tY\mathbf{v}=\sum_{k=0}^jx^{-k}\left(\sum_{i=0}^k\mathbf{v}_i^tY\mathbf{v}_{k-i}\right)\mod x^{-(j+1)}$, there exist $A_j,D_{j-1},B_j$ which satisfy the following equations:
    \begin{equation}
    \begin{split}
        A_0^tY'A_j+A_j^tY'A_0\equiv -\sum_{i=1}^{j-1}A_i^tY'A_{j-i}-\sum_{i=0}^{j-1}C_i^tY''C_{j-i-1}+Z_j^{11}\mod x^{-1}\\
        D_{j-1}^tY''D_0+D_0^tY''D_{j-1}=-\sum_{i=1}^{j-1}\left(D_i^tY''W_{j-i-1}+B_i^tY'B_{j-i}\right)+Z_n^{22}\mod x^{-1}\\
        A_0^tYB_j=-\sum_{i=0}^{j-1}\left(A_i^tY'B_{j-i}+C_i^tY''D_{j-i-1}\right)+Z_j^{12}\mod x^{-1},
    \end{split}
    \end{equation}
    where $Z_j^{11},Z_j^{22},$ and $Z_j^{12}$ are obtained from step 2 and step $j-1$. Hence, there exists $k\in \mathrm{K}$ such that $k^tYk=Y$ and $k\mathbf{e}_1=\mathbf{v}$.
\end{proof}
\bibliography{Ref}
\bibliographystyle{amsalpha}
\end{document}